\documentclass[11pt,a4paper]{article}
\usepackage[left=2.0cm, top=2.45cm,bottom=2.45cm,right=2.0cm]{geometry}
\usepackage{mathtools,amssymb,amsthm,mathrsfs,calc,graphicx,cleveref,dsfont,tikz,pgfplots,bm,url}
\usepackage[british]{babel}
\usepackage{amsfonts}              
\usepackage[T1]{fontenc}
\usepackage{enumitem}
\numberwithin{equation}{section}
\numberwithin{figure}{section}

\usepackage{float}
\restylefloat{table}

\theoremstyle{plain}
\newtheorem{theorem}{Theorem}[section]
\newtheorem{lemma}[theorem]{Lemma}

\newtheorem{proposition}[theorem]{Proposition}

\theoremstyle{definition}
\newtheorem{definition}[theorem]{Definition}

\theoremstyle{remark}
\newtheorem{remark}[theorem]{Remark}

\newcommand{\dint}{\textup{d}}
\newcommand{\Gam}{\mathop{\mathrm{Gamma}}\nolimits}
\newcommand{\Bet}{\mathop{\mathrm{Beta}}\nolimits}
\newcommand{\apex}{\mathop{\mathrm{apex}}\nolimits}

\def\BB{\mathbb{B}}

\def\EE{\mathbb{E}}

\def\MM{\mathbb{M}}
\def\NN{\mathbb{N}}

\def\PP{\mathbb{P}}

\def\RR{\mathbb{R}}
\def\SS{\mathbb{S}}

\def\VV{\mathbb{V}}

\def\XX{\mathbb{X}}

\def\sfN{{\sf N}}

\def\cA{\mathcal{A}}
\def\cB{\mathcal{B}}
\def\cC{\mathcal{C}}
\def\cD{\mathcal{D}}

\def\cF{\mathcal{F}}

\def\cL{\mathcal{L}}

\def\cN{\mathcal{N}}

\def\cT{\mathcal{T}}

\def\cV{\mathcal{V}}

\def\cX{\mathcal{X}}

\newcommand{\dd}{{\rm d}}
\newcommand{\conv}{\mathop{\mathrm{conv}}\nolimits}

\newcommand{\aff}{\mathop{\mathrm{aff}}\nolimits}
\newcommand{\pow}{\mathop{\mathrm{pow}}\nolimits}

\newcommand{\inter}{\operatorname{int}}

\newcommand{\Vol}{\operatorname{Vol}}

\newcommand{\ii}{{\rm{i}}}
\newcommand{\Res}{\mathop{\mathrm{Res}}\nolimits}

\makeatletter
\let\@fnsymbol\@alph
\makeatother

\begin{document}

\title{\bfseries The $\beta$-Delaunay tessellation I:\\ Description of the model and geometry of typical cells}

\author{Anna Gusakova\footnotemark[1],\; Zakhar Kabluchko\footnotemark[2],\; and Christoph Th\"ale\footnotemark[3]}

\date{}
\renewcommand{\thefootnote}{\fnsymbol{footnote}}
\footnotetext[1]{M\"unster University, Germany. Email: gusakova@uni-muenster.de}

\footnotetext[2]{M\"unster University, Germany. Email: zakhar.kabluchko@uni-muenster.de}

\footnotetext[3]{Ruhr University Bochum, Germany. Email: christoph.thaele@rub.de}

\maketitle

\begin{abstract}
\noindent In this paper two new classes of stationary random simplicial tessellations, the so-called $\beta$- and $\beta'$-Delaunay tessellations, are introduced. Their construction is based on a space-time paraboloid hull process and generalizes that of the classical Poisson-Delaunay tessellation. The distribution of volume-power weighted typical cells is explicitly identified, establishing thereby a remarkable connection to the classes of $\beta$- and $\beta'$-polytopes. These representations are used to determine principal characteristics of such cells, including volume moments, expected angle sums and cell intensities.\\

\end{abstract}

\section{Introduction}

Random tessellations in a Euclidean space are among the most central objects studied in stochastic geometry. Their analysis is motivated by their rich inner-mathematical structures, but in equal measure also by the wide range of applications in which they arise. For example, tessellations, and especially triangulations of a space, play a prominent role for finite element methods in numerical analysis, in computer vision, material science, ecology, chemistry, astrophysics, machine learning, network modelling or computational geometry; we refer to the monographs \cite{AurenhammerKlein,BlaszEtAl,ChegDeyEtAl,Edelsbrunner,Haenggi,LoBook,Mo94,OkabeEtAl,PreparataShamos,SW,SKM} as well as the references cited therein for an extensive overview. However, there are only very few mathematically tractable models for which rigorous results are available and which do not require an analysis purely by computer simulations. Among these models are the Poisson-Voronoi tessellations and their duals, the Poisson-Delaunay tessellations. Their construction can be described as follows. Given a stationary Poisson point process $X$ in $\RR^{d-1}$ we define for any point $v\in X$ the Voronoi cell $C(v,X)$ of $v$ as
$$
C(v,X) := \{w\in\RR^{d-1}:\|w-v\|\leq\|w-v'\|\text{ for all }v'\in X\},
$$
that is, $C(v,X)$ contains all points that are closer to $v$ than to any other point from $X$. In applications, $C(v,X)$ might represent the domain of influence of a point $v$, for example the area in a communication network that a base station placed at $v$ may cover. All such Voronoi cells are random convex polytopes and the collection of all Voronoi cells is the Poisson-Voronoi tessellation of $\RR^{d-1}$. To define the dual tessellation we say that the convex hull $\conv(v_1,\ldots,v_d)$ of $d$ distinct points from $X$ is a Delaunay simplex provided that $X$ has no points in the interior of the ball containing $v_1,\ldots,v_d$ on its boundary. The collection of all Delaunay simplices is what is known as the Poisson-Delaunay tessellation. It is easy to observe that $\conv(v_1,\ldots,v_d)$ is a Delaunay simplex if and only if the Voronoi cells of the points $v_1,\ldots,v_d$ meet at a common point which is then the center of the circumscribed sphere of the simplex.

It is the purpose of this series of papers to introduce and to initiate a systematic study of a generalization of Poisson-Delaunay tessellations, the so-called $\beta$-Delaunay tessellations denoted by $\cD_\beta$.  This is a one-parametric family of random tessellations of $\RR^{d-1}$, where the parameter $\beta$ satisfies $-1<\beta<\infty$. We will also introduce the dual tessellations of $\cD_\beta$, which are denoted by $\cV_\beta$ and called the $\beta$-Voronoi tessellations.  The classical Poisson-Delaunay and Poisson-Voronoi tessellations arise as the limiting cases of these tessellations  when $\beta\to -1$.

\begin{figure}[!t]
\centering
\includegraphics[width=0.48\textwidth]{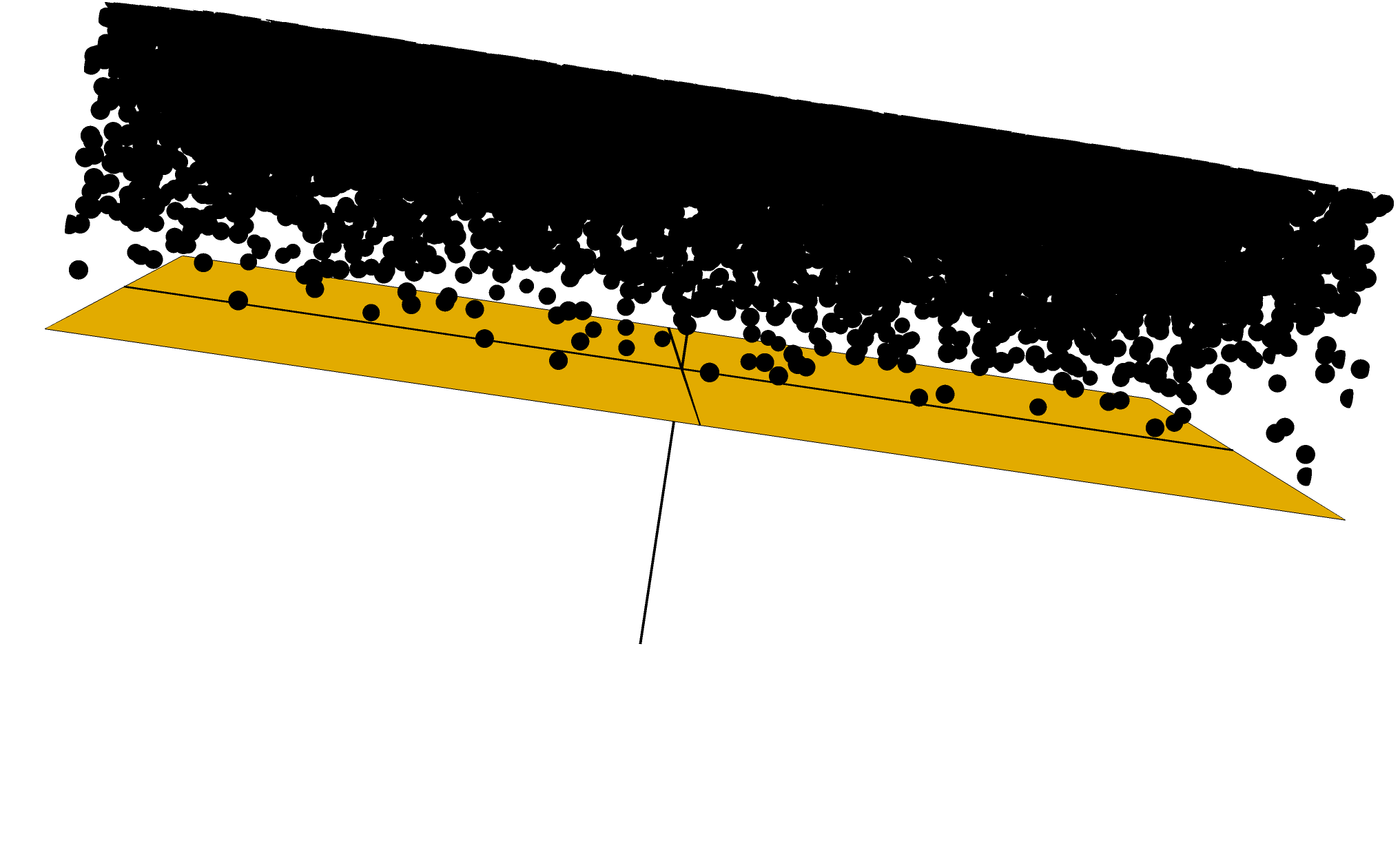}
\includegraphics[width=0.48\textwidth]{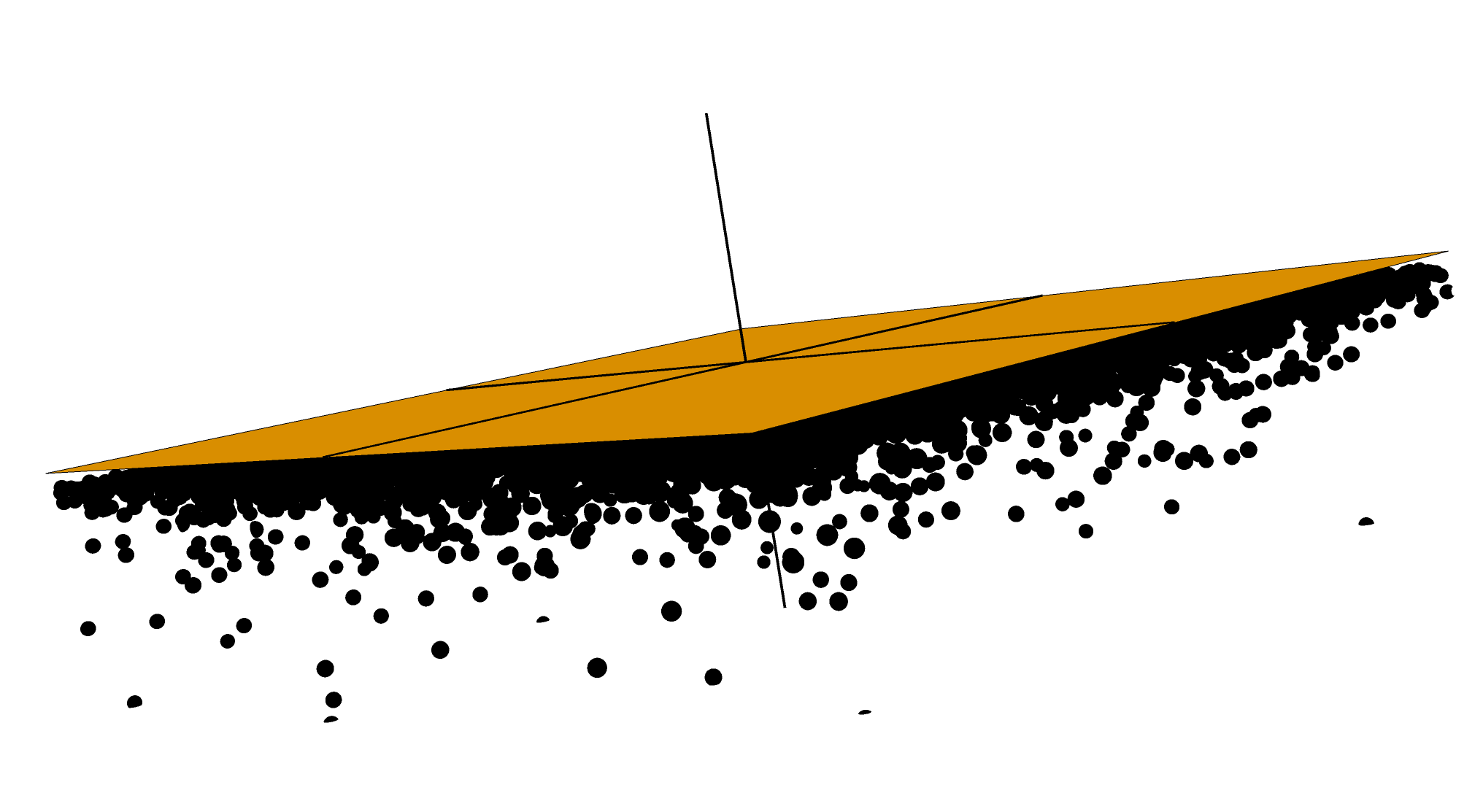}
\caption{Poisson point processes in $\RR^d$ with $d=3$ used to construct the tessellations on the plane. Left: $\eta_\beta$ with $\beta=2$. Right: $\eta_\beta'$ with $\beta=3$. The plane $h=0$ in which the tessellation is constructed is shown in yellow.}
\label{fig:beta_d=3}
\end{figure}

As for the classical Poisson-Delaunay tessellation, the construction of the $\beta$-Delaunay tessellation is based on a Poisson point process as well. However, while the Poisson point process for the Poisson-Delaunay tessellation is located in $\RR^{d-1}$, for the $\beta$-Delaunay tessellation we start with a Poisson point process $\eta_{\beta}$ in the product space $\RR^{d-1}\times [0,\infty)$ whose intensity measure has the form $\text{const}\cdot h^\beta\,\dint v\dint h$, where $v\in\RR^{d-1}$ stands for the spatial coordinate and $h>0$ for the height coordinate of a point $x=(v,h)\in\RR^{d-1}\times [0,\infty)$. A realization of this Poisson point process for $d=3$ is shown in the left panel of Figure~\ref{fig:beta_d=3}. In a next step, we construct the paraboloid hull process associated with $\eta_{\beta}$. This is a particular germ-grain process with paraboloid grains which in stochastic geometry was introduced by Schreiber and Yukich \cite{SY08}, and further developed in Calka, Schreiber and Yukich \cite{CSY13} and Calka and Yukich~\cite{CYGaussian} in order to study the asymptotic geometry of random convex hulls near their boundary. We shall use the same paraboloid hull process to construct a random tessellation $\cD_\beta$ of $\RR^{d-1}$ with only simplicial cells  as follows. Given $d$ points $x_1=(v_1,h_1),\ldots,x_d=(v_d,h_d)$ of $\eta_\beta$ with affinely independent spatial coordinates $v_1,\ldots,v_d$, there is a unique shift of the standard downward paraboloid
$$
\Pi^\downarrow := \left\{(v,h)\in\RR^{d-1}\times\RR\colon h\leq -\|v\|^2\right\}
$$
containing $x_1,\ldots,x_d$ on its boundary. We declare $\conv(v_1,\ldots,v_d)$ to be a $\beta$-Delaunay simplex in $\RR^{d-1}$ if and only if the interior of the downward paraboloid determined by $x_1,\ldots,x_d$ is void of points of $\eta_{\beta}$. The collection of all $\beta$-Delaunay simplices is called the $\beta$-Delaunay tessellation of $\RR^{d-1}$. Two realizations of the paraboloid hull process  in the case $d=2$ are shown in Figure~\ref{fig:beta_d=2}. It should be mentioned at this point that one can alternatively think of $\cD_{\beta}$ as the dual to the Laguerre diagram $\cV_\beta$ of $\eta_{\beta}$, where each point $x=(v,h)\in\eta_{\beta}$ represents a sphere with center $v$ and (imaginary) radius $\sqrt{-h}$; see Section~\ref{sec:lag_diag} for more details. We will call $\cV_\beta$ the $\beta$-Voronoi tessellation. The tessellations $\cD_\beta$ and $\cV_\beta$ are dual to each other in the following sense.  A simplex $\conv(v_1,\ldots,v_d)$ is a cell of the $\beta$-Delaunay tessellation if and only if the cells generated by $x_1,\ldots,x_d$ in the $\beta$-Voronoi tessellation are non-empty and meet at a common point. Both constructions of the $\beta$-Delaunay tessellation are equivalent and each of them plays its role in studying the properties of $\cD_{\beta}$.

{Let us present a simple description of the $\beta$-Voronoi tessellation $\cV_\beta$.}  Imagine that each atom $x=(v,h)\in \RR^{d-1}\times [0,\infty)$ of the Poisson point process $\eta_\beta$ gives rise to a crystallization process in $\RR^{d-1}$ which starts at spatial position $v\in \RR^{d-1}$ at time $h>0$. The speed of the crystallization process is not constant and assumed to be such that the process reaches a point $w\in \RR^d$ at time $h + \|w-v\|^2$. Then, the cell generated by $(v,h)$ is just the set of all points $w\in \RR^{d-1}$ that are reached by the crystallization process started at $(v,h)$ not later than by a crystallization process started at any other point $(v',h')$. It should be emphasized that the cell may be empty and that in the case when it is non-empty, it need not contain the point $v$ (which is different from the case of the classical Poisson-Voronoi tessellation). The set of cells with non-empty interior forms the $\beta$-Voronoi tessellation $\cV_\beta$.

\begin{figure}[!t]
\centering
\includegraphics[width=0.98\textwidth]{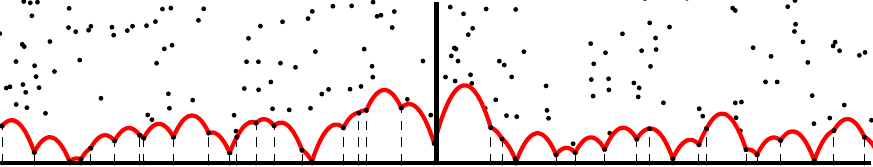}\\
\vspace{0.5cm}
\includegraphics[width=0.98\textwidth]{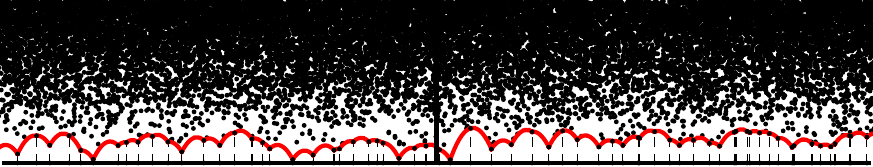}
\caption{Construction of the $\beta$-Delaunay tessellation $\cD_\beta$ for $d=2$. The figure shows the Poisson point process $\eta_{\beta}$ and the corresponding paraboloid hull. Top: $\beta=0$, bottom: $\beta=2$. Points on the horizontal axis are the vertices of the $\beta$-Delaunay tessellation of $\RR$.}
\label{fig:beta_d=2}
\end{figure}

The goal of  the present paper (which is the first in a series of papers) is to show that these constructions do indeed lead to  well-defined stationary random tessellation of $\RR^{d-1}$ and to study their geometric properties. Moreover and in parallel to the construction of the $\beta$-Delaunay tessellation we introduce the concept of a $\beta'$-Delaunay tessellation $\cD_{\beta}'$ (together with its dual $\beta'$-Voronoi tessellation $\cV_\beta'$), whose construction is based on a Poisson point process $\eta_\beta'$ on the product space $\RR^{d-1}\times (-\infty,0)$ with intensity measure $\text{const}\cdot (-h)^{-\beta}\,\dint v\dint h$ for $\beta>(d+1)/2$. A realization of the Poisson point process $\eta_\beta'$ for $d=3$ is shown on the right panel of Figure~\ref{fig:beta_d=3}, while the paraboloid hull process in the case $d=2$ is shown in Figure~\ref{fig:betaprime_d=2}.

Whenever possible, we will develop our results for both, $\beta$ and $\beta'$, random tessellation models in parallel. In particular, we are interested in what is known as the typical cell of $\cD_\beta$ and $\cD_\beta'$. Intuitively, one can think of such a cell as a randomly chosen cell, which is selected independently of its size and shape. More generally, we shall study volume-power weighted typical cells, where the weight is certain  power $\nu$ of the volume. One of our main contributions is {Theorem \ref{theo:typical_cell_stoch_rep}, which gives} a precise distributional characterization of the weighted typical cell of $\cD_\beta$ and $\cD_\beta'$. In particular, we will prove that the weighted typical cell of $\cD_\beta$ and $\cD_\beta'$ is a randomly rescaled volume-power weighted $\beta$- or $\beta'$-simplex, respectively {(see Remark \ref{rem:rep_typical}), which also explains the names of the models}. {Very remarkably, this provides a new link between the $\beta$-and $\beta'$-Delaunay tessellation and the class of $\beta$-polytopes and $\beta'$-polytopes, which was under intensive investigation in recent times, see, for example, \cite{GKT17,kabluchko_formula,KTT,beta_polytopes}, and also Remark \ref{rem:BetaPolytopes} below. It opens a way to study the geometry and the combinatorial properties of the typical cells of $\cD_\beta$ and $\cD_\beta'$.} Among our results are explicit formulas for the moments of the volume as well as probabilistic representations in terms of independent gamma- and beta-distributed random variables. We also compute explicitly the $j$-face intensities and determine the expected angle sums of weighted typical cells.

Finally, we prove that, as $\beta\to\infty$, the expected angle sums of the volume-power weighted typical cells tend to those of a regular simplex in $\RR^{d-1}$. We will pick up this topic in detail in part II of this series of papers, where we describe the common limiting tessellation $\cD:=\cD_{\infty}$ of $\cD_\beta$ and $\cD_\beta'$, as $\beta\to\infty$, after suitable rescaling. This will provide an explanation of the limit behaviour of the expected angle sums just described. In part III we will prove various high-dimensional limit theorems for the volume of weighted typical cells in $\cD_\beta$, $\cD_\beta'$ and $\cD$, that is, limit theorems where $d\to\infty$ (potentially in a coupled way with other parameters). We also describe there the shape of large weighted typical cells in the spirit of Kendall's problem, generalizing thereby results of Hug and Schneider \cite{HugSchneiderDelaunay} on the classical Poisson-Delaunay tessellation. {Finally, in part IV we study mixing properties of the tessellations $\cD_\beta$, $\cD_\beta'$ and $\cD$ and, as a consequence, we derive central limit theorems for the number of $k$-dimensional faces of $\cD_\beta$ and $\cD$ in a growing window.}

\begin{figure}[!t]
	\centering
	\includegraphics[width=0.98\textwidth]{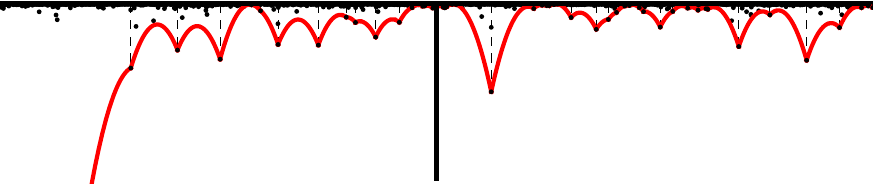}\\
	\vspace{0.5cm}
	\includegraphics[width=0.98\textwidth]{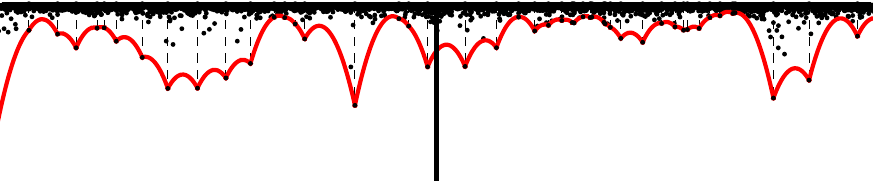}
	\caption{Construction of the $\beta'$-Delaunay tessellation $\cD_\beta'$ for $d=2$. The figure shows the Poisson point process $\eta_{\beta}'$ and the corresponding paraboloid hull. Top: $\beta=2$, bottom: $\beta=3$.}
	\label{fig:betaprime_d=2}
\end{figure}

\medspace

The remaining parts of this text are structured as follows. In Section \ref{sec:Preliminaries} we recall the necessary notions and notation from random tessellation and point process theory, which are used throughout the paper.  The detailed construction of $\beta$-Delaunay tessellations is presented in Section \ref{sec:Construction}. The explicit distributions of their  volume-power weighted typical cells is the content of Section \ref{sec:TypicalCells}. These results are used in Section \ref{sec:Volume} to derive explicit formulas and probabilistic representations for the moments of the volume of such cells. The final Section \ref{sec:AnglesFaceIntensities} discusses expected angle sums of weighted typical cells as well as formulas for face intensities.

\section{Preliminaries}\label{sec:Preliminaries}

\subsection{Frequently used notation}

Let $d\geq 1$ and $A\subset\RR^d$. We denote by ${\rm int}\,A$ the interior of $A$ and by $\partial A$ its boundary. A centred closed Euclidean ball in $\RR^d$ with radius $r>0$ is denoted by $\BB_r^d$ and we put $\BB^d:=\BB_1^d$. The volume of $\BB^d$ is given by
$$
\kappa_d:=\frac{\pi^{d/2}}{\Gamma(1+{d\over 2})}.
$$
By $\sigma_{d-1}$ we denote the spherical Lebesgue measure on $(d-1)$-dimensional unit sphere $\SS^{d-1}=\partial \BB^d$, normalized in such a way that
$$
\omega_d:=\sigma_{d-1}(\SS^{d-1})={2\pi^{d\over 2}\over \Gamma\left({d\over 2}\right)}.
$$

Given a set $C\subset \RR^{d-1}$ denote by $\conv(C)$ its convex hull. For points $v_0,\ldots,v_k\in\RR^{d-1}$ we write $\aff(v_0,\ldots,v_k)$ for the affine hull of $v_0,\ldots,v_k$, which is at most $k$-dimensional affine subspace of $\RR^{d-1}$, $k\in\{0,1,\ldots,d-1\}$.

In what follows we shall represent points $x\in\RR^d$ in the form $x=(v,h)$ with $v\in\RR^{d-1}$ (called \textit{spatial coordinate}) and $h\in\RR$ (called \textit{height}, \textit{weight} or \textit{time coordinate}).

We will often use the following notation. Let $\Pi$ (respectively, $\Pi^+$) be the standard downward  (respectively, upward) paraboloid, defined as
\begin{align*}
\Pi
&:=
\{(v',h')\in\RR^{d-1}\times\RR\colon h'=-\|v'\|^2\},\\
\Pi^+
&:=
\{(v',h')\in\RR^{d-1}\times\RR\colon h'=\|v'\|^2\}.
\end{align*}
Further, let $\Pi_{x}$ be the translation of $\Pi$ by a vector $x:=(v,h)\in\RR^d$, that is,
$$
\Pi_x:=\{(v',h')\in\RR^{d-1}\times\RR\colon h'=-\|v'-v\|^2+h\}.
$$
Moreover, given a set $A\subset \RR^d$ we define the hypograph and the epigraph of $A$ as
\begin{align*}
A^{\downarrow}:&=\{(v,h')\in\RR^{d-1}\times\RR\colon (v,h) \in A \text{ for some } h\ge h'\},\\
A^{\uparrow}:&=\{(v,h')\in\RR^{d-1}\times\RR\colon (v,h) \in A \text{ for some } h\leq h'\}.
\end{align*}
The point $x$ is the apex of the paraboloid $\Pi_x$ and we write
$
\apex\Pi^{\downarrow}_x=\apex\Pi_x:=x.
$

\subsection{(Poisson) point processes}

Let $(\XX,\cX)$ be a measurable space supplied with a $\sigma$-finite measure $\mu$. By $\sfN(\XX)$ we denote the space of $\sigma$-finite counting measures on $\XX$. The $\sigma$-field $\cN(\XX)$ is defined as the smallest $\sigma$-field on $\sfN(\XX)$ such that the evaluation mappings $\xi\mapsto\xi(B)$, $B\in\cX$, $\xi\in\sfN(\XX)$ are measurable. A \textbf{point process} on $\XX$ is a measurable mapping with values in $\sfN(\XX)$ defined over some fixed probability space $(\Omega,\cA,\PP)$. By a \textbf{Poisson point process} $\eta$ on $\XX$ with intensity measure $\mu$ we understand a point process with the following two properties:
\begin{itemize}
	\item[(i)] for any $B\in\cX$ the random variable $\eta(B)$ is Poisson distributed with mean $\mu(B)$;
	\item[(ii)] for any $n\in\NN$ and pairwise disjoint sets $B_1,\ldots,B_n\in\cX$ the random variables $\eta(B_1),\ldots,\eta(B_n)$ are independent.
\end{itemize}
We refer to \cite{LP,SW} for the existence and construction of Poisson point processes and for further details.

\subsection{Tessellations}

In this subsection we recall the concept of a random tessellation and include a brief overview of basic properties. For more detailed discussion we refer reader to \cite[Chapter 10]{SW}. Roughly speaking, a tessellation (or a mosaic) is a system of polytopes that cover the whole space and have disjoint interiors. We fix a space dimension $d\geq 2$. Since the tessellations we construct in Section \ref{sec:Construction} are driven by a Poisson point process in $\RR^d$ and induce a tessellation in $\RR^{d-1}$, we consider this set-up in what follows.

\begin{definition}\label{def:Tessellation}
A \textbf{tessellation} $M$ in $\RR^{d-1}$ is a countable system of subsets of $\RR^{d-1}$ satisfying the following conditions:
\begin{enumerate}[label=\alph*)]
\item $M$ is locally finite system of non-empty closed sets, where local finiteness means that every bounded subset of $\RR^{d-1}$ has non-empty intersection of only finitely many sets from $M$;
\item the sets $m\in M$ are compact, convex and have interior points;
\item the sets of $M$ cover the space, meaning that
$$
\bigcup\limits_{m\in M}m=\RR^{d-1};
$$
\item if $m_1,m_2\in M$ and $m_1\neq m_2$, then $\inter m_1\cap \inter m_2=\varnothing$.
\end{enumerate}
\end{definition}

The elements of $M$ are called {\bf cells} of $M$ and they are convex polytopes by \cite[Lemma~10.1.1]{SW}. Given a polytope $P$ we denote for $k\in\{0,1,\ldots,d-1\}$ by $\cF_k(P)$ the set of its $k$-dimensional faces and let $\cF(P):=\bigcup_{k=0}^{d-1}\cF_{k}(P)$. A tessellation $M$ is called {\bf face-to-face} if for all $P_1,P_2\in M$ we have
$$
P_1\cap P_2\in(\cF(P_1)\cap\cF(P_2))\cup\{\varnothing\}.
$$
A face-to-face tessellation in $\RR^{d-1}$ is called {\bf normal} if each $k$-dimensional face of the tessellation is contained in precisely $d-k$ cells for all $k\in\{0,1,\ldots,d-2\}$. We denote by $\MM$ the set of all face-to-face tessellations in $\RR^{d-1}$. By a {\bf random tessellation} in $\RR^{d-1}$ we understand a particle process $X$ in $\RR^{d-1}$ (in the usual sense of stochastic geometry, see \cite{SW}) satisfying $X\in\MM$ almost surely. Implicitly, we assume here and for the rest of this paper that all the random objects we consider are defined on a probability space $(\Omega,\cA,\PP)$.

\section{Construction of $\beta^{(')}$-Voronoi and $\beta^{(')}$-Delaunay tessellations}\label{sec:Construction}

\subsection{Underlying point processes}

Let us start by defining the $\beta$- or $\beta'$-Voronoi tessellation and its dual, the $\beta$- or $\beta'$-Delaunay tessellation.  We will give two alternative definitions using the concept of Laguerre tessellations and notion of paraboloid hull process introduced by Schreiber and Yukich \cite{SY08}, Calka, Schreiber and Yukich \cite{CSY13} and Calka and Yukich~\cite{CYGaussian}. As an underlying process for the $\beta$-Voronoi and the $\beta$-Delaunay tessellations we consider a space-time Poisson point process $\eta=\eta_{\beta}$ in $\RR^{d-1}\times \RR_{+}$, where $\RR_{+}:=[0,+\infty)$ denotes the set of non-negative real numbers, with intensity measure having density
\begin{equation}\label{eq:BetaPoissonIntensity}
(x,h)\mapsto\gamma\,c_{d,\beta} \cdot h^{\beta},
\qquad
c_{d,\beta}:={\Gamma\left({d\over 2}+\beta+1\right)\over \pi^{d\over 2}\Gamma(\beta+1)},
\qquad
\gamma > 0,
\,
\beta>-1,
\end{equation}
with respect to the Lebesgue measure on $\RR^{d-1}\times \RR_{+}$. Here, $\gamma>0$ is the intensity parameter which  usually will be suppressed in our notation, $\beta>-1$ is the shape parameter, and the normalizing constant $c_{d,\beta}$ has been introduced to simplify some of the computations below.  See Figure~\ref{fig:beta_d=3} (left panel) and Figure~\ref{fig:beta_d=2} for realizations of these point processes with $d=3$ and $d=2$, respectively.

For the $\beta'$-Voronoi and the $\beta'$-Delaunay tessellations we consider a space-time Poisson point process $\eta'=\eta^{\prime}_{\beta}$ in $\RR^{d-1}\times\RR_{-}^*$, where $\RR_{-}^*:=(-\infty,0)$ denotes the set of negative real numbers, with intensity measure having density
\begin{equation}\label{eq:BetaPrimePoissonIntensity}
(x,h)\mapsto\gamma\,c'_{d,\beta} \cdot (-h)^{-\beta},
\qquad
c'_{d,\beta}:={\Gamma\left(\beta\right)\over \pi^{d\over 2}\Gamma(\beta-{d\over 2})},
\qquad
\gamma > 0,
\,
\beta >{d+1\over 2},
\end{equation}
with respect to the Lebesgue measure on $\RR^{d-1}\times\RR_{-}^*$. Again, $\gamma>0$ and $\beta>(d+1)/2$ are the parameters of the process. See Figure~\ref{fig:beta_d=3} (right panel) and Figure~\ref{fig:betaprime_d=2} for realizations of $\eta^{\prime}_{\beta}$ with $d=3$ and $d=2$, respectively.

As we will see later, the $\beta$- and $\beta'$-tessellations can often be treated in a unified way.  To make this explicit and in order to shorten and simplify the presentation of the paper we introduce the following variable notation. We put $\kappa:=1$ if the $\beta$-model is considered and $\kappa:=-1$ if we work with the $\beta'$-model. Moreover, through the paper we will use almost the same notation for $\beta$- and $\beta'$-tessellations which will only differ by using the $\prime$-symbol in case of $\beta'$ model. When the formulas for the $\beta$- and $\beta'$-model are close enough to join them into one expression, we will indicate this by $\!\!\!\!\phantom{x}^{(\prime)}$, meaning that this sign should be omitted as far as a $\beta$-model is considered.

Using the convention just introduced we can consistently represent the density of the Poisson point process $\eta^{(')}$ on $\RR^{d-1}\times\RR$ as
$$
(x,h)\mapsto\gamma\,c_{d,\beta}^{(')} \cdot (\kappa h)^{\kappa\beta} \cdot {\bf 1}\{\kappa h>0\}
$$
with $\beta>-1$ for the $\beta$-model and $\beta>(d+1)/2$ for the $\beta'$-model.

\subsection{General Laguerre diagrams}\label{sec:lag_diag}

Let us start by defining a Laguerre tessellation, which can be considered as a generalized (or weighted) version of a classical Voronoi tessellation. For two points $v,w \in \RR^{d-1}$ and $h\in\RR$ we define the power of $w$ with respect to the pair $(v,h)$ as
\[
\pow (w,(v,h)):=\|w-v\|^2+h.
\]
In this situation $h$ is referred to as the weight of the point $v$. A closely related concept is known from elementary geometry, where for $h < 0$ the value $\pow(w, (v, h))$ describes the square length of the tangent through a point $w$ at a circle with radius $\sqrt{-h}$ around $v$.

Let $X$ be a countable set of marked points of the form $(v,h)$ in $\RR^{d-1}\times \RR$.
Then the \textbf{Laguerre cell} of $(v,h)\in X$ is defined as
\[
C((v,h),X):=\{w\in\RR^{d-1}\colon \pow(w,(v,h))\leq \pow(w,(v',h'))\text{ for all }(v',h')\in X\}.
\]
Let us mention an intuitive interpretation of the notions introduced above in terms of crystallization process. Imagine that $v\in \RR^{d-1}$ denotes a point at which certain crystallization process starts at time $h\in\RR$. Then, $X$ is a collection of centres of crystallization together with the corresponding initial times. Suppose further that after a crystallization process has started at some point $v\in\RR^{d-1}$, it needs time $R^2$ to cover a ball of radius $R>0$ around $v$. In particular, the spreading speed of crystallization is non-constant and decreases with time. Then, $\pow (w,(v,h))$ is just the time at which the point $w$ is covered by the crystallization process that started at $(v,h)$. Moreover, the Laguerre cell of $(v,h)$ is just the crystal with ``center'' $v$, that is the set of points which are covered by the crystallization process that started at  $(v,h)$ before they are covered by any other crystallization process. It should be pointed out that in our model we assume that crystallization starts at point $(v,h)\in X$ even if this point is already covered by another crystallization process, which started earlier.

Note that Laguerre cells can have vanishing topological interior and, in our case, most of them will actually be empty. The collection of all Laguerre cells of $X$, which have non-vanishing topological interior, is called the \textbf{Laguerre diagram}:
\[
\cL(X):=\{C((v,h),X)\colon (v,h)\in X, C((v,h),X)\neq\varnothing\}.
\]
Also, let us emphasize that the Laguerre cell generated by $(v,h)$, even if it is non-empty,  need not contain the point $v$. Indeed, the cell of $(v,h)$ does not contain $v$ if at time $h$ the point $v$  has been already covered by a crystallization process that started at some other point $(v',h')\neq (v,h)$.

Laguerre diagrams play an important role in computational geometry, where it is more common to use the following interpretation of $\cL(X)$ as a vertical projection of a $d$-dimensional polyhedral set to $\RR^{d-1}$, see \cite[Chapters~17,18]{BY98}. Assuming that $h=-r^2$ for some complex number $r\in\mathbb{C}$, to every point $(v,-r^2)\in X$ we assign  a $(d-2)$-dimensional sphere $S(v,r)\subset \RR^{d-1}$  with center $v\in\RR^{d-1}$ and (potentially imaginary) radius $r$. Two spheres $S(v_1,r_1)$ and $S(v_2,r_2)$ are orthogonal if $\|v_1-v_2\|^2=r_1^2+r_2^2$ or, equivalently, $\pow(v_2,(v_1,-r_1^2))=r_2^2$. Thus, for a point $w\in\RR^{d-1}$, the value $\pow(w,(v,-r^2))$ is the squared radius of the sphere centred at $w$ and orthogonal to $S(v,r)$. Consider the transformation $\phi: S(v,r)\mapsto (v,\|v\|^2-r^2)$ that maps $(d-2)$-dimensional spheres to points in $\RR^{d}$. Considering a point $w\in\RR^{d-1}$ as a sphere of zero radius, the above transformation maps $\RR^{d-1}$ to the standard upward paraboloid $\Pi^+$ in $\RR^{d}$. Moreover, each vertical line in $\RR^{d}$ passing through $(v,0)$ represents the set of spheres centred at $v$, and real spheres are mapped to points below $\Pi^+$, while imaginary spheres correspond to points above $\Pi^+$. Treating the paraboloid $\Pi^+$ as a quadric, for each point $x\in\RR^{d}$ we can define the polar hyperplane $x^o$ of $x$ with respect to $\Pi^+$. Then \cite[Lemma~17.2.1]{BY98} implies that the set of spheres orthogonal to a given sphere $S(v,r)$ is mapped by $\phi$ to $\phi(S(v,r))^o$. Moreover, \cite[Lemma~17.2.3]{BY98} shows that $\pow(w,(v,-r^2))$ is the signed vertical distance between the point $\phi(w)$ and the hyperplane $\phi(S(v,r))^o$, i.e.\ the difference of the $d$-th coordinates of $\phi(w)$ and the unique point $y\in\phi(S(v,r))^o$ with $(y_1,\ldots,y_{d-1})=w$. Using this interpretation and writing
$$
P(X):=
\bigcap_{(v,h)\in X}\phi(S(v,\sqrt{-h}))^{o,\uparrow}
$$
it is clear that the Laguerre diagram $\cL(X)$ arises as the vertical projection (setting the $d$-th coordinate to zero) of the boundary of $P(X)$ to $\RR^{d-1}$. Thus, each non-empty Laguerre cell is a vertical projection of a face of $P(X)$. For a more detailed overview of the above interpretation we refer the reader to \cite[Chapters 17,18]{BY98}.

\subsection{Random Laguerre tessellations}\label{sec:Laguerre_tess}

It should be mentioned that a Laguerre diagram is not necessarily a tessellation, at least as long as we do not impose additional assumptions on the geometric properties of the set $X$. We also note that the case when all weights are equal corresponds to the case of the classical Voronoi tessellation. The first formal description of geometric properties {of the set $X$ ensuring that the Laguerre diagram $\cL(X)$ becomes a tessellation} is due to Schlottmann \cite{Sch93} and a thorough investigation of the case when all weights are negative has been made by Lautensack and Zuyev \cite{LZ08, Ldoc}. In our situation we are interested in the case of positive, negative, as well as general weights $h\in\RR$. More precisely and in view of the applications in part II of this series of papers, we consider a point process $\xi$ in $\RR^{d-1}\times E$, where $E\subset \RR$ is a possibly unbounded interval, satisfying the following properties.

\begin{itemize}
\item[(P1)] For every $(w,t)\in\RR^{d-1}\times E$ there are almost surely only finitely many $(v,h)\in\xi$ satisfying
$$
\pow (w,(v,h)) = \|w-v\|^2+h \leq t.
$$
In words, at the time when a crystallization process starts at some $(w,t)$, the point $w$ is already reached by at most finitely many crystallization processes.
\item[(P2)] With probability $1$ we have
$$
\conv(v\colon (v,h)\in\xi)=\RR^{d-1}.
$$

\item[(P3)] With probability $1$ no $d+1$ points $(v_0,h_0),\ldots, (v_d,h_d)$ from $\xi$ lie on the same downward paraboloid of the form
    $$
    \{(v,h)\in \RR^{d-1}\times E:  \|v - w\|^2 + h = t\}
    $$
with  $(w,t)\in\RR^{d-1}\times E$. In words, with probability $1$ it is not possible that $d+1$ crystallization processes reach the same point in space simultaneously.
\end{itemize}

In the following two lemmas we prove that the Laguerre diagram constructed on the Poisson point process $\xi$ is a random face-to-face normal tessellation in $\RR^{d-1}$.

\begin{lemma}
Let $\xi$ be a point process satisfying conditions (P1) and (P2). Then  $\cL(\xi)$ is a random face-to-face tessellation in $\RR^{d-1}$.
\end{lemma}

\begin{proof}
The proof follows directly from \cite[Proposition 1]{Sch93}. In this context, it should be noted that in \cite{Sch93} the function
$$
\pow^*(w,(v,h)) = \|w-v\|^2-h =\|w-v\|^2+(-h)=\pow(w,(v,-h))
$$
was considered. However, this does not influence the proof and we prefer to use the function $\pow(\,\cdot\,)$ in this paper, which is more convenient for our purposes.
\end{proof}

\begin{lemma}\label{lem:voronoi_normal}
Let $\xi$ be a point process satisfying conditions (P1)--(P3). Then the random tessellation $\cL(\xi)$ is normal with probability $1$.
\end{lemma}

\begin{proof}
{The statement follows directly from the interpretation of the Laguerre diagram $\cL(\xi)$ as the  boundary of the polyhedral set $P(\xi)$. Let $\xi'$ be the image of the point process $\xi$ under the transformation $(v,h)\mapsto (v,\|v\|^2+h)$. Condition (P3) is equivalent to the fact that with probability $1$ no $d+1$ points $(w_0,s_0),\ldots, (w_d,s_d)$ from $\xi'$ are lying on the same hyperplane in $\RR^{d}$. From this it also follows that for any $k=1,\ldots, d$ with probability $1$ no $k+1$ points $(w_0,s_0),\ldots, (w_k,s_k)$ are lying on the same $(k-1)$-dimensional hyperplane. Thus, using polarity with respect to the paraboloid $\Pi^+$, we conclude that any $(d-k)$-dimensional face of $P(\xi)$, $k=1,\ldots,d$, is contained in exactly $k$ facets, which implies the normality of $\cL(\xi)$.}
\end{proof}

\begin{remark}
It is easy to see that if the point process $\xi$ satisfies properties (P1) --- (P3) then with probability $1$ there is no non-empty Laguerre cell with vanishing interior. To see this, we again use the correspondence between the Laguerre tessellation $\cL(\xi)$ and the boundary of $P(\xi)$. First of all we note that every Laguerre cell  with vanishing interior corresponds to a face $F$ of $P(\xi)$ with dimension $k\leq d-2$. Thus, there are at least $d-k$ facets of $P(\xi)$ containing $F$ and hence using the polarity relation between hyperplanes with respect to paraboloid $\Pi^+$ there are at least $d-k+1$ points in $\xi'$ lying on the same $(d-1-k)$-dimensional hyperplane. According to (P3) this happens with probability $0$.
\end{remark}

\subsection{Definition of $\beta^{(')}$-Voronoi tessellations}

The $\beta^{(')}$-Voronoi tessellations we are interested in are defined as {random} Laguerre tessellations driven by the Poisson point processes $\eta_\beta$ and  $\eta^{\prime}_\beta$ with intensities given by~\eqref{eq:BetaPoissonIntensity} and~\eqref{eq:BetaPrimePoissonIntensity}.

\begin{lemma}\label{lem:properties_satisfied}
The point processes $\eta_\beta$ for $\beta>-1$ and  $\eta^{\prime}_\beta$ for $\beta>(d+1)/2$ satisfy properties (P1)--(P3) with $E=[0,\infty)$ in the $\beta$-case and $E= (-\infty,0)$ in the $\beta'$-case.
\end{lemma}
\begin{proof}
Property (P2) holds because the projections of the Poisson point processes $\eta_\beta$ and $\eta_\beta^\prime$ to the space component $\RR^{d-1}$ are everywhere dense sets, with probability $1$. Indeed, the integrals of the intensities of these processes over any set of the form $B\times \RR$, where $B\subset \RR^{d-1}$ is a non-empty ball, are infinite, which means that infinitely many points project to $B$ almost surely.

Property (P3) holds for any Poisson point process $\eta$ in $\RR^{d-1}\times E$ whose intensity measure $\Theta_{\eta}$ is absolutely continuous with respect to the Lebesgue measure with density $\varrho$, say (see, for example, \cite[Proposition 4.1.2]{Mo94} for a closely related result in the setting of a stationary Poisson point process). Let us verify (P3) applying the multivariate Mecke formula \cite[Corollary 3.2.3]{SW}. Write $\eta_{\neq}^{d+1}$ for the set of $(d+1)$-tuples of distinct points of $\eta$ and $\Pi(x_1,\ldots,x_d)$ for the unique downward paraboloid on which points $x_1,\ldots,x_d\in\RR^d$ are lying. With this notation we have that
\begin{align*}
\EE&\sum_{(x_1,\ldots,x_{d+1})\in \eta^{d+1}_{\neq}}{\bf 1}(x_1,\ldots,x_{d+1}\text{ lie on the same paraboloid})\\
&={1\over (d+1)!}\int_{\RR^{d-1}}\ldots \int_{\RR^{d-1}}\int_{E}\ldots \int_{E}\;\;\prod_{i=1}^{d}\varrho(h_i,v_i)\,\dd v_1\ldots \dd v_{d}\,\dd h_1\ldots \dd h_{d}\\
&\qquad\times \int_{\RR^{d-1}}\int _{E} {\bf 1}((v_{d+1}, h_{d+1})\in\Pi((v_1,h_1),\ldots,(v_d,h_d)))\varrho(h_{d+1},v_{d+1})\,\dd v_{d+1}\,\dd h_{d+1} = 0,
\end{align*}
since the inner integral is equal to $0$.

Verification of property (P1) requires additional computations.  To consider the case of $\eta_\beta$, fix $w\in\RR^{d-1}$ and $t>0$. The inequalities in (P1) describe the bounded domain
$$
D := \{(v,h)\in \RR^{d-1}\times \RR_+: \|v-w\|^2 + h \leq t\}
$$
lying below the paraboloid $h = t-\|v-w\|^2$ and above the hyperplane $h=0$.  Since the intensity measure of the Poisson point process $\eta_\beta$ is locally integrable due to the condition $\beta>-1$, there are only finitely many points $(v,h)$ of $\eta_\beta$ in $D$  and Property (P1) holds.

In order to check (P1) for $\eta^{\prime}_\beta$, we fix $w\in\RR^{d-1}$ and $t<0$. We need to show that the downward paraboloid
$$
D: = \{(v,h)\in \RR^{d-1}\times \RR_-^*: \|v-w\|^2 + h \leq t\}
$$
contains only finitely many points of $\eta_\beta^\prime$ a.s. Using the stationarity of the process $\eta_\beta^\prime$ in the space coordinate, we can put $w=0$ without loss of generality.   The expected number of  points of $\eta_\beta^\prime$ in $D$ is then given by
\begin{align*}
\EE \sum_{(v,h)\in\eta^{\prime}_\beta} {\bf 1}(\|v\|^2+h\leq t)
&=
\gamma\,c_{d,\beta}^{\prime}
\int_{\RR^{d-1}}\int_{0}^\infty {\bf 1}(\|v\|^2-s\leq t) s^{-\beta}\dd s\,\dd v\\
&=
\frac{\gamma\,c_{d,\beta}^{\prime}}{1-\beta}
\int_{\RR^{d-1}} (\|v\|^2 + |t|)^{1-\beta} \dd v\\
&={\gamma \, c_{d,\beta}^{\prime}\over  1-\beta}\,(d-1)\kappa_{d-1}\int_0^\infty(r^2+|t|)^{1-\beta}r^{d-2}\,\dint r\\
&=
{\gamma \, c_{d,\beta}^{\prime}\over  (1-\beta)c_{d-1,\beta-1}^{\prime}}|t|^{d+1-2\beta\over 2}<\infty,
\end{align*}
where introduced spherical coordinates in $\RR^{d-1}$, applied the definition of the constants $\kappa_{d-1}$ and $c_{d-1,\beta-1}^{\prime}$, and used the condition that $\beta> (d+1)/2$ to ensure the finiteness of the integral over $\RR^{d-1}$. This completes the proof.
\end{proof}

Summarizing, we conclude that the Laguerre tessellations $\cL(\eta_\beta)$ and $\cL(\eta^{\prime}_\beta)$ are with probability one stationary and normal random tessellations in $\RR^{d-1}$. We can thus state the following definition.

\begin{definition}
The random tessellation $\cV_\beta:=\cL(\eta_\beta)$ is called the \textbf{$\beta$-Voronoi tessellation} and the random tessellation $\cV^{\prime}_\beta:=\cL(\eta^{\prime}_\beta)$ is called the \textbf{$\beta'$-Voronoi tessellation} in $\RR^{d-1}$.
\end{definition}

Let us emphasize that even though the Poisson point processes $\eta_\beta$, respectively $\eta^{\prime}_{\beta}$, are actually well-defined on $\RR^{d-1}\times (0,\infty)$, respectively $\RR^{d-1}\times (-\infty,0)$, for every $\beta\in \RR$,  the corresponding tessellations are well-defined under conditions $\beta>-1$, respectively $\beta>(d+1)/2$, only (because otherwise condition (P1) is not satisfied).

\subsection{Definition of  $\beta^{(')}$-Delaunay tessellations}

Given a Laguerre diagram $\cL(\xi)$ we can associate to it a so-called dual Laguerre diagram $\cL^*(\xi)$, which can be defined in the same spirit as a classical Delaunay diagram for given Voronoi construction. 

Let $\xi$ be a Poisson point process in $\RR^{d-1}\times E$, $E\subset \RR$ satisfying properties (P1) --- (P3). Then $\cL(\xi)$ is a random normal face-to-face tessellation and we denote by $\cF_0(\cL(\xi))$ the set of its vertices. Further, given a point $z\in \cF_0(\cL(\xi))$ we construct a Delaunay cell $D(z,\xi)$ as a convex hull of those $v$ for which $(v,h)\in\xi$ and $z\in C((v,h),\xi)$, namely
$$
D(z,\xi): = \conv(v\colon (v,h)\in\xi, z\in C((v,h),\xi)).
$$
Since the tessellation $\cL(\xi)$ is normal with probability $1$, for every vertex $z\in \cF_0(\cL(\xi))$ there exists exactly $d$ points $x_1,\ldots, x_d$ of $\xi$ such the corresponding cells $C(x_1,\xi),\dots, C(x_d,\xi)$ of the Laguerre tessellation $\cL(\xi)$ contain $z$. Thus, $D(z,\xi)$ is a simplex with probability $1$. We define the \textbf{dual Laguerre diagram} $\cL^*(\xi)$ as a collection of all Delaunay simplices
$$
\cL^*(\xi):=\{D(z,\xi)\colon z\in \cF_0(\cL(\xi))\}.
$$
From the above construction it follows that for any $z\in \cF_0(\cL(\xi))$ there exists a number $K_{z}\in\RR$ such that with probability $1$ there exist exactly $d$ points $(v_1,h_1),\ldots, (v_d,h_d)$ of $\xi$ with
$$
\pow(z, (v_1,h_1)) = \ldots = \pow(z, (v_d,h_d)) = K_{z}
$$
and there is no $(v,h)\in\xi$ with $\pow(z,(v,h))<K_{z}$. Consider the set
\begin{equation}\label{eq:ApexProcess}
\xi^*:=\left\{(z,-K_{z})\in\RR^{d-1}\times\RR\colon z\in \cF_0(\cL(\xi))\right\}.
\end{equation}
It turned out to be that the dual Laguerre diagram $\cL^*(\xi)$ is a Laguerre diagram constructed for the set $\xi^*$ and that $\xi^*$ satisfies properties (P1) and (P2) if $\xi$ satisfies (for the proof of those facts see \cite[Proposition 2]{Sch93}). Thus, by Lemma \ref{lem:properties_satisfied} we conclude that $\cL^*(\xi)=\cL(\xi^*)$ is random face-to-face simplicial tessellation.

{As for the usual Laguerre diagram $\cL(X)$, where $X$ is some countable set, there is an alternative construction of the dual Laguerre diagram $\cL^*(X)$, which is more common in computational geometry, see \cite[Chapter 17,18]{BY98}. It can be described as follows. Interpreting points of $(v,h)\in X$ as spheres with center $v$ and radius $\sqrt{-h}\in\mathbb{C}$ and applying the transformation $\phi$ introduced in Section \ref{sec:lag_diag} we obtain a set of points $X':=\{(v,\|v\|^2+h)\colon (v,h)\in X\}$ in $\RR^{d}$. Then the dual Laguerre diagram $\cL^*(X)$ is the vertical projection along the $d$-th coordinate of the boundary of the polyhedral set $P^*(X'):=\conv(X')^{\uparrow}$. It is also common in computational geometry to call $\cL^*(X)$ a "regular triangulation".
}

\begin{figure}[!t]
\centering
\includegraphics[width=0.48\textwidth]{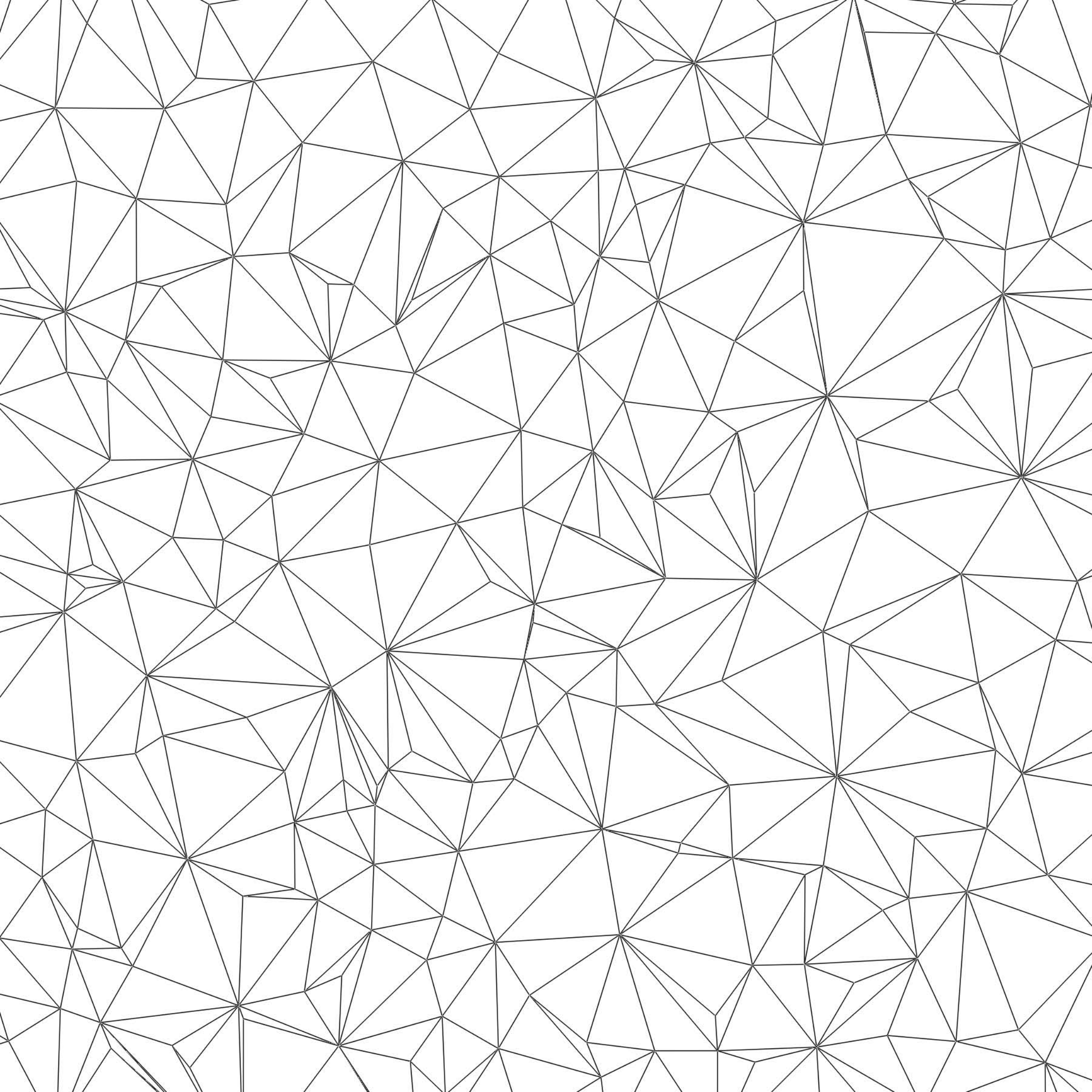}\hspace{0.2cm}
\includegraphics[width=0.48\textwidth]{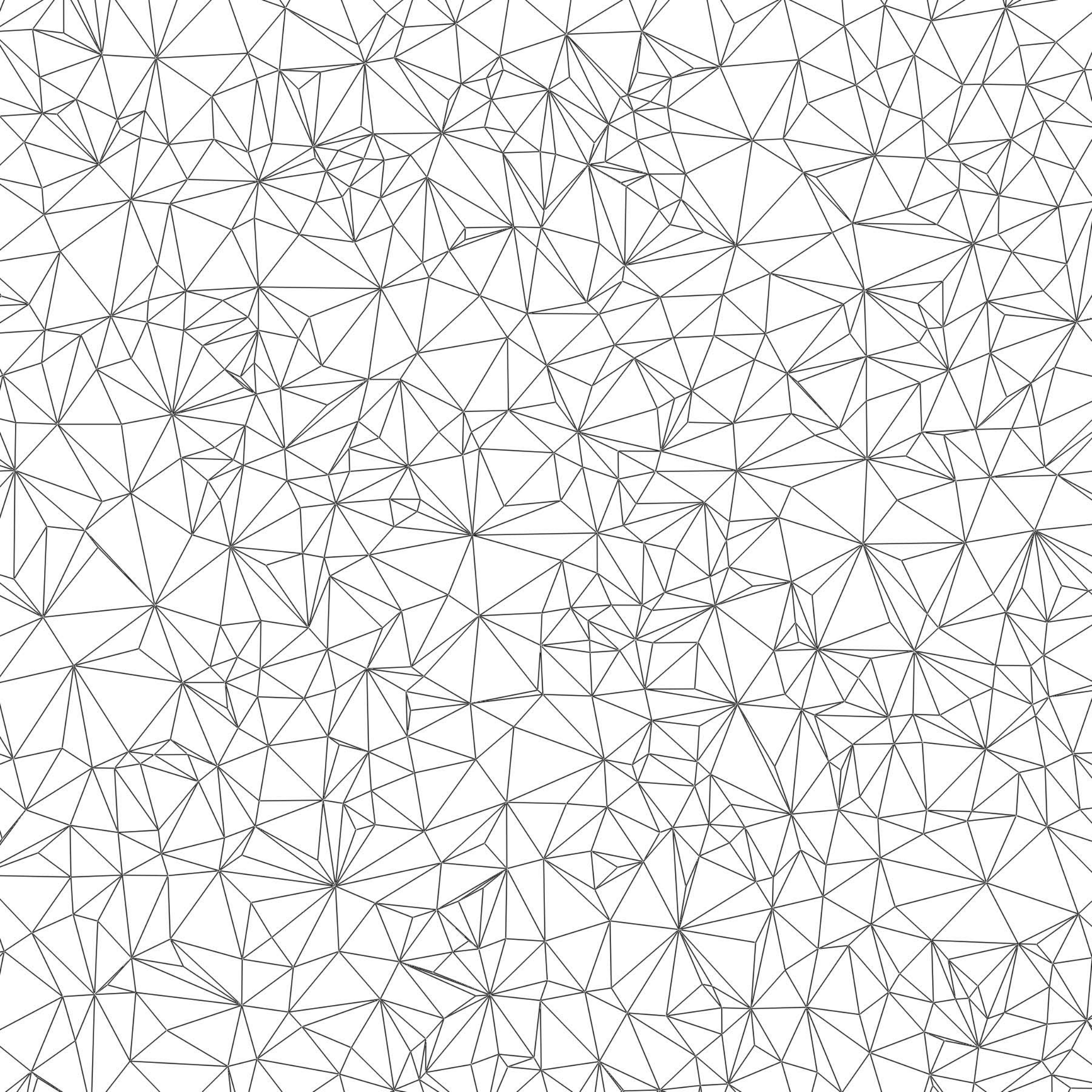}
\caption{Realization of $\beta$-Delaunay tessellation in $\RR^2$. Left: $\beta=5$. Right: $\beta=15$. The pictures above have been created with the help of the software project "The Computational Geometry Algorithms Library" (CGAL) \cite{CGAL}.}
\label{fig:beta-tessellations}
\end{figure}

We will be interested in the case when $\xi$ is one of the Poisson point processes $\eta_\beta$ or $\eta_{\beta}^\prime$.

\begin{definition}
The random tessellation $\cD_\beta:=\cL^*(\eta_{\beta})$ is called the \textbf{$\beta$-Delaunay tessellation} in $\RR^{d-1}$, while the random tessellation $\cD^{\prime}_\beta:=\cL^*(\eta^{\prime}_{\beta})$ is called the \textbf{$\beta'$-Delaunay tessellation} in $\RR^{d-1}$.
\end{definition}

\subsection{Paraboloid hull process}\label{sec:ParabHullProc}

The paraboloid hull process was first introduced in \cite{SY08} and  \cite{CSY13} in order to study the asymptotic geometry of the convex hull of Poisson point processes in the unit ball. It is designed to exhibit properties analogous to those of convex polytopes with the paraboloids playing the role of hyperplanes, with the spatial coordinates $v$ playing the role of spherical coordinates and with the height coordinates $h$ playing the role of the radial coordinate. The numerous properties of the paraboloid hull process, which are analogous to standard statements of convex geometry, have been developed in \cite[Section 3]{CSY13} and we refer to this paper for further information and background material. At this point let us mention without making the statement precise and proving it, that the $\beta$-Delaunay tessellation we are interested in describes the local asymptotic structure (near the boundary of the unit sphere) of the so-called beta random polytope~\cite{beta_polytopes} in the $d$-dimensional unit ball generated by $n$ points, as $n\to\infty$. After rescaling, the unit sphere looks locally like $\RR^{d-1}$, while the boundary of the beta random polytope (projected to the sphere) looks locally like the $\beta$-Delaunay tessellation.

\begin{figure}[!t]
\centering
\includegraphics[width=0.48\textwidth]{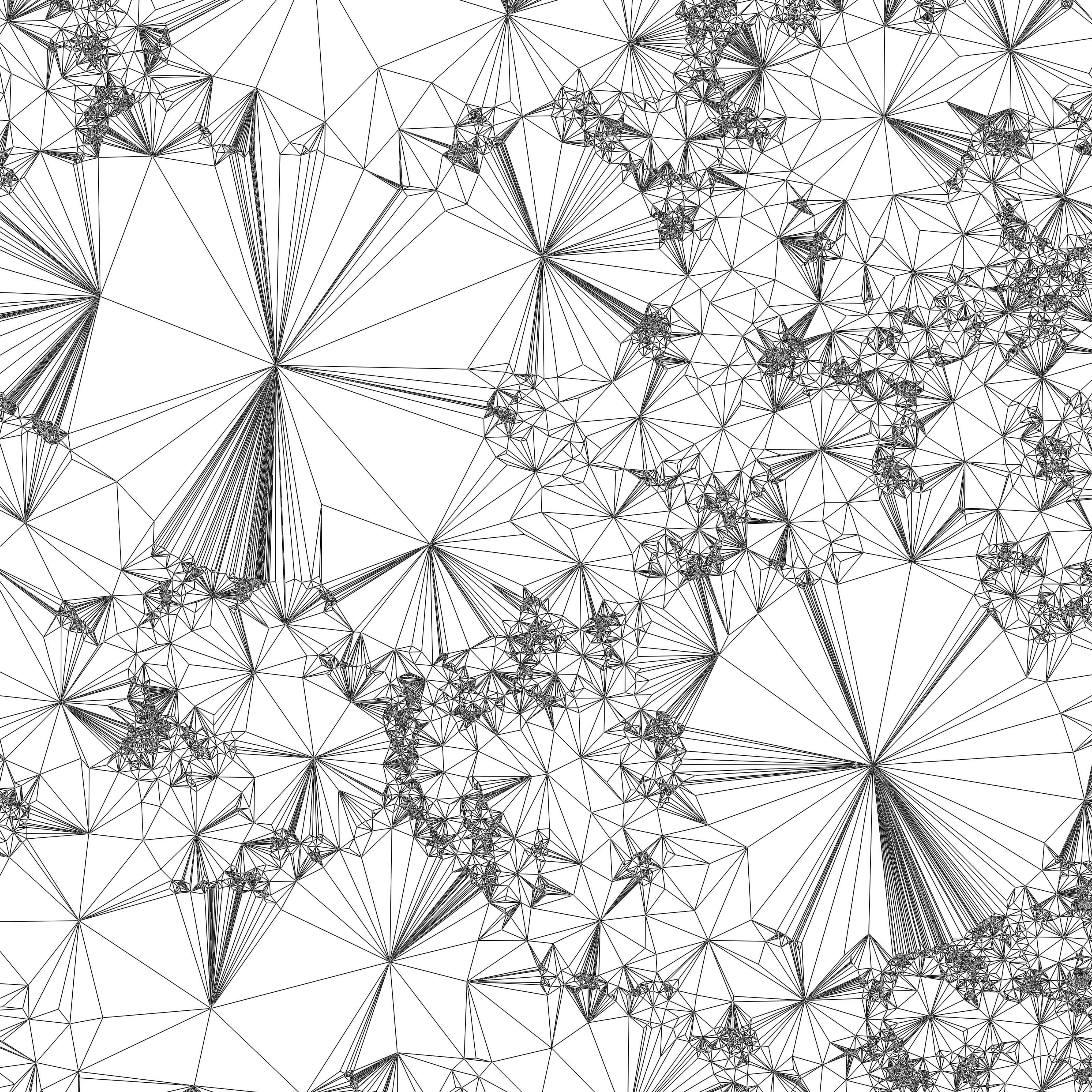}\hspace{0.2cm}
\includegraphics[width=0.48\textwidth]{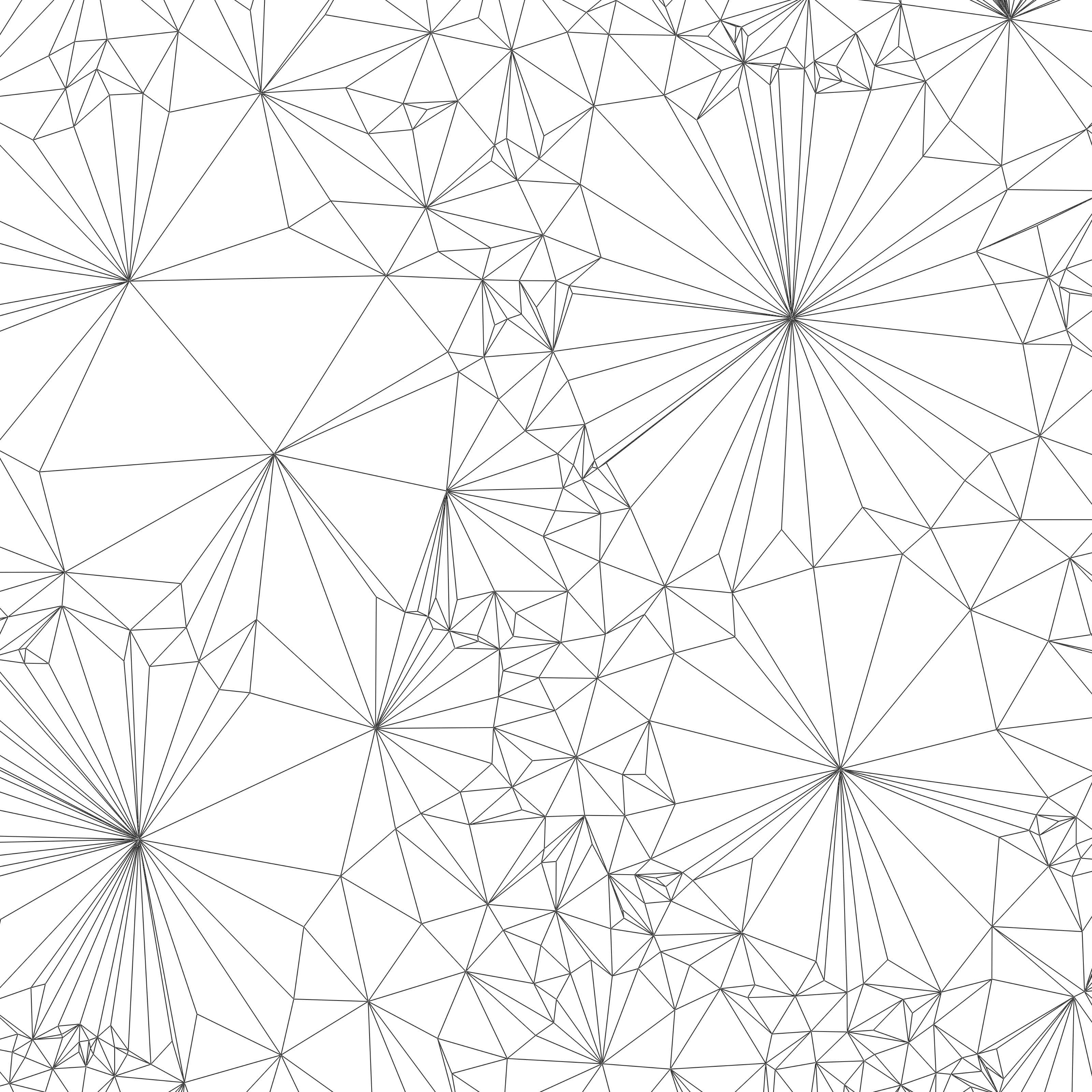}
\caption{Realization of $\beta^{'}$-Delaunay tessellation in $\RR^2$. Left: $\beta=2.1$. Right: $\beta=2.5$. The pictures above have been created with the help of the software project "The Computational Geometry Algorithms Library" (CGAL) \cite{CGAL}.}
\label{fig:betaprime-tessellations}
\end{figure}

The idea is that the shifts of the hypographs of standard paraboloid $\Pi^{\downarrow}$ are, in some sense, analogous to the half-spaces in $\RR^{d}$ not containing the origin $0$ in their boundary. For any collection $x_1:=(v_1,h_1),\ldots,x_k:=(v_k,h_k)$ of $k\leq d$ points in $\RR^{d-1}\times\RR$ with affinely independent coordinates $v_1,\ldots,v_k$, we define $\Pi(x_1,\ldots,x_k)$ as the intersection of $\aff(v_1,\ldots,v_k)\times\RR$ with a translation of $\Pi$ containing all points $x_1,\ldots,x_k$. It should be noted that the set $\Pi(x_1,\ldots,x_k)$ is well-defined, although for $k<d$ the translation of $\Pi$ containing all $x_1,\ldots,x_k$ is not unique. Nevertheless, for $k=d$ and all
tuples $x_1,\ldots,x_d$ with affinely independent spatial coordinates $v_1,\ldots,v_d$ such a translation is unique. Then we define $\Pi[x_1,\ldots,x_k]$ as
\[
\Pi[x_1,\ldots,x_k]:=\Pi(x_1,\ldots,x_k)\cap \left(\conv(v_1,\ldots,v_k)\times\RR\right).
\]

We will say that a set $A\subset \RR^d$ has the paraboloid convexity property if for each $y_1,y_2\in A$ we have $\Pi[y_1,y_2]\subset A$. Clearly, $\Pi[x_1,\ldots,x_k]$ is the smallest set containing $x_1,\ldots,x_k$ and having the paraboloid convexity property. Next, we say the set $A\subset \RR^d$ is upwards paraboloid convex if and only if $A$ has the paraboloid convexity property and if for each $x=(v,h)\in A$ we have $\{x\}^{\uparrow}\subset A$.

Finally, given a locally finite point set $X\subset\RR^d$ we define its \textbf{paraboloid hull} $\Phi(X)$ to be the smallest upwards paraboloid convex set containing $X$. In particular, given a Poisson point process $\xi$ in $\RR^{d-1}\times E$, $E\subset\RR$, we define the \textbf{paraboloid hull process} $\Phi(\xi)$ in $\RR^{d-1}\times E$ as the paraboloid hull of $\xi$.

Using the arguments analogous to \cite[Lemma 3.1]{CSY13} it is easy to derive an alternative and more convenient way to represent $\Phi(\xi)$, namely with probability $1$ we have that
\[
\Phi(\xi)=\bigcup\limits_{(x_1,\ldots,x_d)\in\xi_{\neq}^d}\left(\Pi[x_1,\ldots, x_d]\right)^{\uparrow},
\]
where $\xi_{\neq}^d$ is the collection of all $d$-tuples of distinct points of $\xi$.

For $(x_1,\ldots,x_d)\in\xi_{\neq}^d$ the set $\Pi[x_1,\ldots,x_d]$ is called a paraboloid sub-facet of $\Phi(\xi)$ if $\Pi[x_1,\ldots,x_d] \subset\partial \Phi(\xi)$. Two paraboloid sub-facets $\Pi[x_1,\ldots,x_d]$ and $\Pi[y_1,\ldots,y_d]$ are called co-paraboloid provided that $\Pi(x_1,\ldots,x_d)=\Pi(y_1,\ldots,y_d)$ and by \textbf{paraboloid facet} of $\Phi(\xi)$ we understand the collection of co-paraboloid sub-facets. Since $\xi$ is a Poisson point process each paraboloid facet of $\Phi(\xi)$ with probability one consists of exactly one sub-facet. Thus, we can say, that $\Pi[x_1,\ldots,x_d]$ is a paraboloid facet of $\Phi(\xi)$ if and only if $\xi \cap \left(\Pi(x_1,\ldots,x_d)\right)^{\downarrow}=\{x_1,\ldots, x_d\}$.

Using the paraboloid hull processes $\Phi(\xi)$ we construct now a diagram $\cD(\xi)$ on $\RR^{d-1}$ in the following way: for any collection $x_1:=(v_1,h_1),\ldots,x_d:=(v_d,h_d)$ of pairwise distinct points from $\xi$ we say that the simplex $\conv(v_1,\ldots,v_d)$ belongs to $\cD(\xi)$ if and only if $\Pi[x_1,\ldots,x_k]$ is a paraboloid facet of $\Phi(\xi)$. Thus, if $\xi$ satisfies properties (P1)--(P3), then $\cD(\xi)=\cL^{*}(\xi)$ is a random simplicial tessellation.

It is clear now that the tessellations $\cD(\eta_{\beta})$ and $\cD(\eta^{\prime}_{\beta})$ coincide with $\beta$-Poisson-Delaunay and $\beta'$-Poisson tessellations (respectively), defined in the previous subsection.

\section{Weighted typical cells in $\beta$- and $\beta^{\prime}$-Delaunay  tessellations}\label{sec:TypicalCells}

\subsection{Definition of the $\nu$-weighted typical cell}

In this section we derive an explicit representation of the distribution of typical cells in a $\beta$-Delaunay tessellation $\cD_\beta$ with parameter $\beta > -1$ and a $\beta^{\prime}$-Delaunay  tessellation $\cD^{\prime}_\beta$ with parameter $\beta > (d+1)/2$ as described in the previous sections, and, more generally, the distribution of typical cells weighted by the $\nu$-th power of their volume, with $\nu\ge-1$.
On the intuitive level, the construction presented below can be understood as follows. Consider the collection of all cells of $\cD_\beta$ or $\cD^{\prime}_\beta$ and  assign to each cell a weight equal to the $\nu$-th power of its volume. Then, pick one cell at random, where the probability of picking each cell is proportional its $\nu$-th volume power. The resulting random simplex is denoted by $Z_{\beta,\nu}$, respectively $Z_{\beta,\nu}^\prime$, and its probability distribution on the space of compact convex subsets of $\RR^{d-1}$ is denoted by $\PP_{\beta,\nu}$, respectively $\PP_{\beta,\nu}^\prime$.
Since the number of cells in the tessellation is infinite, some work is necessary in order to define these objects in a mathematically rigorous way. The reader should keep  in mind the following two important special cases:
\begin{itemize}
\item[(i)] $\nu=0$: $Z_{\beta,0}$ and $Z_{\beta,0}^{\prime}$ coincide with the classical typical cell of $\cD_\beta$ and $\cD_\beta^{\prime}$, respectively;
\item[(ii)] $\nu=1$:  $Z_{\beta,1}$ and $Z_{\beta,1}^{\prime}$ coincide the volume-weighted typical cell of $\cD_\beta$ and $\cD_\beta^{\prime}$, respectively (which has the same distribution as the a.s.\ unique cell containing the origin, up to translation; see Theorem~10.4.1 in~\cite{SW}).
\end{itemize}

To formally present the definition of volume-power weighted typical cells, we use the concept of generalized centre functions and Palm calculus for marked point processes as outlined in \cite[p.\ 116]{SW} and \cite[Section 4.3]{SWGerman}. Following the arguments from Subsection \ref{sec:Laguerre_tess} and Subsection \ref{sec:ParabHullProc}, a random tessellation $\cD(\xi)$, where $\xi$ is a point process in $\RR^{d-1}\times E$, $E\subset \RR$ satisfying (P1)--(P3), coincides with the Laguerre tessellation of the random set $\xi^*$ described by \eqref{eq:ApexProcess}. In this section we additionally assume that $\xi$ is stationary with respect to the shifts of the $\RR^{d-1}$-component, which implies stationarity of the tessellation $\cD(\xi)$. Observe that $\xi^*$ can alternatively be described via the set of apexes of paraboloid facets of the paraboloid hull process $\Phi(\xi)$, that is,
\[
 \xi^*=\{(v,h)\colon (v,-h)=\apex\Pi(x_1,\ldots,x_d),\, x_i\in\xi,1\leq i\leq d,\,\conv(v_1,\ldots,v_d)\in\cD(\xi)\}.
\]
Let $\cC'$ denote the space of non-empty compact subsets of $\RR^{d-1}$ endowed with the usual Hausdorff metric and the corresponding Borel $\sigma$-field $\cB(\cC')$. The random tessellation $\cD(\xi)$ (which is defined as a random subset of $\cC'$) can be identified with the particle process $\sum_{c\in\cD(\xi)}\delta_c$, see \cite[Chapter 4]{SW}. Formally, this is a simple point process on $\cC'$, or equivalently, a random element in the space $\sfN_s(\cC')$ of $\sigma$-finite simple counting measures on $\cC'$ (a counting measure $\zeta$ on $\cC'$ is simple if $\zeta(\{K\})\in\{0,1\}$ for all $K\in\cC'$). Next, we define the measurable set $\cC'\circ\sfN(\cC'):=\{(K,\zeta)\in\cC' \times \sfN_s(\cC'):K\in\zeta\}$ and recall that a generalized centre function is any Borel-measurable map $z:\cC'\circ\sfN(\cC')\to\RR^{d-1}$ such that $z(K+t, \zeta+t) = z(K,\zeta)+t$ for every $t\in\RR^{d-1}$ and $(K,\zeta)\in\cC'\circ\sfN(\cC')$. In our case we take
$$
z(K,\zeta):=\begin{cases}
v &: K=C((v,h), \xi^*),\,\zeta ={\cD}(\xi)\\
0 &: \text{otherwise},
\end{cases}
$$
where we recall that $C((v,h), \xi^*)$ is the Laguerre cell of $(v,h)\in \xi^*$.

In a next step, we consider the random marked point process $\mu_{\xi}$ on $\RR^{d-1}$ with mark space $\cC'$, formed as follows:
\[
\mu_{\xi}:=\sum\limits_{(v,h)\in \xi^*}\delta_{(v,M)},\qquad M:=C((v,h), \xi^*)-v.
\]
It is evident from the construction that the point process $\mu_{\xi}$ is stationary and that the intensity measure $\Theta$ of $\mu_{\xi}$ is locally finite. Thus, according to \cite[Theorem 3.5.1]{SW} it admits the decomposition
\[
\Theta=\lambda\,[{\rm Leb}(\RR^{d-1})\otimes\PP_{\xi,0}],
\]
where $0<\lambda<\infty$, ${\rm Leb}(\RR^{d-1})$ is the Lebesgue measure on $\RR^{d-1}$ and $\PP_{\xi,0}$ is a probability measure on $\cC'$, the so-called mark distribution of $\mu_{\xi}$. By \cite[p.\ 84]{SW} it can be represented as
$$
\PP_{\xi,0}(A) := {1\over\lambda}\EE\sum_{(v,M)\in\mu_\xi}{\bf 1}_A(M){\bf 1}_{[0,1]^{d-1}}(v),
$$
where $[0,1]^{d-1}$ denotes the $(d-1)$-dimensional unit cube. The probability measure $\PP_{\xi,0}$ describes the mark attached to the typical point of $\mu_{\xi}$, that is, the typical cell of the tessellation $\cD(\xi)$. This motivates the following definition.

For a given $\nu$ we define a probability measure $\PP_{\xi,\nu}$ on $\cC'$ by
\begin{equation}\label{eq_2}
\PP_{\xi,\nu}(A) := {1\over \lambda_{\xi,\nu}}\EE\sum_{(v,M)\in\mu_{\xi}}{\bf 1}_A(M){\bf 1}_{[0,1]^{d-1}}(v)\Vol(M)^\nu
\end{equation}
for $A\in  \cB(\cC')$, where  $\lambda_{\xi,\nu}$ is the normalizing constant given by
\begin{equation}\label{eq:def_gamma_const}
\lambda_{\xi,\nu}:= \EE\sum_{(v,M)\in\mu_{\xi}}{\bf 1}_{[0,1]^{d-1}}(v)\Vol(M)^\nu.
\end{equation}
It should be mentioned, that for some values of $\nu$ the value $\lambda_{\xi,\nu}$  can be equal to infinity. This is the reason why for any point process $\xi$ one needs to specify possible values of $\nu$.

\begin{definition}
A random simplex $Z_{\beta,\nu}$, where $\nu \ge -1$ and $\beta>-1$, with distribution $\PP_{\beta,\nu}:=\PP_{\eta_\beta,\nu}$ is the {\bf  $\Vol^{\nu}$-weighted} (or just {\bf $\nu$-weighted}) {\bf typical cell} of the $\beta$-Delaunay tessellation $\cD_\beta$.
\end{definition}
\begin{definition}
A random simplex $Z_{\beta,\nu}^{\prime}$, where $\beta>(d+1)/2$ and $2\beta - d>\nu\ge -1$, with distribution $\PP_{\beta,\nu}^{\prime}:=\PP_{\eta_\beta^{\prime},\nu}$ is the {\bf $\Vol^{\nu}$-weighted typical cell} of the $\beta^{\prime}$-Delaunay tessellation $\cD^{\prime}_\beta$.
\end{definition}

\begin{remark}
That the constants $\lambda_{\beta,\nu}:=\lambda_{\eta_{\beta},\nu}$ and $\lambda^{\prime}_{\beta,\nu}:=\lambda_{\eta_{\beta}^{\prime},\nu}$ are in fact finite for the ranges of $d$, $\beta$ and $\nu$ mentioned in the previous definition will be established in the proof of Theorem~\ref{theo:typical_cell_stoch_rep}.
\end{remark}

\begin{remark}
We also conjecture that it is possible to enlarge the diapason of possible values for parameter $\nu$ to $\nu>-2$.
\end{remark}

\subsection{Stochastic representation of the $\nu$-weighted typical cell}

After having introduced the concept of weighted typical cells, we are now going to develop an explicit description of their distributions. In fact, the following theorem may be considered as our main contribution in this paper, since it is the principal device on which most of the results in this part, but also in part II and III of this series of papers are based on. To present it, let us recall our convention that $\kappa=1$ if we consider the $\beta$-model and that $\kappa=-1$ in case of the $\beta'$-model.

\begin{theorem}\label{theo:typical_cell_stoch_rep}
Fix $d\geq 2$, $\nu\ge-1$ and $\beta>-1$ for the $\beta$-model or $2\beta - d>\nu\ge -1$, $\beta>(d+1)/2$ for the $\beta^{\prime}$-model. Then for any Borel set $A\subset \cC'$ we have that
\begin{align*}
\PP_{\beta,\nu}^{(\prime)}(A)
&=
\alpha_{d,\beta,\nu}^{(\prime)}\int_{(\RR^{d-1})^d}\dd y_1\ldots \dd y_d \, \int_{0}^{\infty}\dd r\,{\bf 1}_A(\conv(ry_1,\ldots,ry_d)) r^{2\kappa d\beta+d^2+\nu(d-1)}\notag\\
&\qquad\times e^{-m^{(\prime)}_{d,\beta} r^{d+1+2\kappa\beta}}
\Delta_{d-1}(y_1,\ldots,y_d)^{\nu+1}\prod_{i=1}^d(1-\kappa\|y_i\|^2)^{\kappa\beta}{\bf 1}(1-\kappa\|y_i\|^2\ge 0),
\end{align*}
where $\Delta_{d-1}(y_1,\ldots,y_d)$ is the volume of $\conv(y_1,\ldots,y_d)$ and $\alpha_{d,\beta,\nu}$, $\alpha^{\prime}_{d,\beta,\nu}$, $m_{d,\beta}$  and $m^{\prime}_{d,\beta}$ are constants given by
\begin{align}
m_{d,\beta}^{(\prime)}&=\gamma\,c_{d,\beta}^{(\prime)}(\pi c_{d+1,\beta}^{(\prime)})^{-1},\label{eq:Constant}\\
\alpha_{d,\beta,\nu}&=\pi^{d(d-1)\over 2}\,{(d-1)!^{\nu+1}(d+1+2\beta)\Gamma({d(d+\nu+2\beta)-\nu+1\over 2})\over \Gamma({d(d+\nu+2\beta)\over 2}+1)\Gamma(d+{(\nu-1)(d-1)\over d+2\beta+1})}\Big({\gamma\,\Gamma({d\over 2}+\beta+1)\over \sqrt{\pi}\Gamma({d+1\over 2}+\beta+1)}\Big)^{d+{(\nu-1)(d-1)\over d+2\beta+1}}\notag\\
&\qquad\qquad\times{\Gamma({d+\nu\over 2}+\beta+1)^d\over \Gamma(\beta+1)^d}\prod\limits_{i=1}^{d-1}{\Gamma({i\over 2})\over \Gamma({i+\nu+1\over 2})},\label{eq:Alpha}\\
\alpha^{\prime}_{d,\beta,\nu}&=\pi^{d(d-1)\over 2}\,{(d-1)!^{\nu+1}(d+1-2\beta)\Gamma({d(2\beta-d-\nu)\over 2})\over \Gamma({d(2\beta-d-\nu)+\nu+1\over 2})\Gamma(d+{(\nu-1)(d-1)\over d-2\beta+1})}\Big({\gamma\,\Gamma(\beta-{d+1\over 2})\over \sqrt{\pi}\Gamma(\beta-{d\over 2})}\Big)^{d+{(\nu-1)(d-1)\over d-2\beta+1}}\notag\\
&\qquad\qquad\times{\Gamma(\beta)^d\over \Gamma(\beta-{d+\nu\over 2})^d}\prod_{i=1}^{d-1}{\Gamma({i\over 2})\over \Gamma({i+\nu+1\over 2})}.\label{eq:AlphaPrime}
\end{align}
\end{theorem}
\begin{remark}\label{rem:rep_typical}
In more probabilistic terms, the $\nu$-weighted typical cell of the $\beta$-Delaunay tessellation $\cD_\beta$ has the same distribution as the random simplex $\conv(RY_1,\ldots,RY_d)$, where
\begin{enumerate}
\item[(a)] $R$ is a random variable whose density is proportional to $r^{2d\beta+d^2+\nu(d-1)}e^{-m_{d,\beta} r^{d+1+2\beta}}$ on $(0,\infty)$;
\item[(b)] $(Y_1,\ldots,Y_d)$ are $d$ random points in the unit  ball $\BB^{d-1}$ whose joint density is proportional to
$$
\Delta_{d-1}(y_1,\ldots,y_d)^{\nu+1}   \prod\limits_{i=1}^d(1-\|y_i\|^2)^{\beta},
\qquad y_1\in \BB^{d-1},\ldots, y_d\in \BB^{d-1};
$$
\item[(c)] $R$ is independent of $(Y_1,\ldots,Y_d)$.
\end{enumerate}
In other words this means that the $\nu$-weighted typical cell of the $\beta$-Delaunay tessellation coincides in distribution with the randomly rescaled $(\nu+1)$-volume-weighted $\beta$-simplex $\conv(Y_1,\ldots,Y_d)$.

In the same way, the $\nu$-weighted typical cell of the $\beta^{\prime}$-Delaunay tessellation $\cD^{\prime}_\beta$ has the same distribution as the random simplex $\conv(RY_1,\ldots,RY_d)$, where
\begin{enumerate}
\item[(a$^\prime$)] $R$ is a random variable whose density is proportional to $r^{-2d\beta+d^2+\nu(d-1)}e^{-m^{\prime}_{d,\beta} r^{d+1-2\beta}}$ on $(0,\infty)$;
\item[(b$^\prime$)] $(Y_1,\ldots,Y_d)$ are $d$ random points in $\RR^{d-1}$ whose joint density is proportional to
$$
\Delta_{d-1}(y_1,\ldots,y_d)^{\nu+1}   \prod\limits_{i=1}^d(1+\|y_i\|^2)^{-\beta},
\qquad y_1\in \RR^{d-1},\ldots, y_d\in \RR^{d-1};
$$
\item[(c$^\prime$)] $R$ is independent of $(Y_1,\ldots,Y_d)$.
\end{enumerate}
Alternatively we say that the $\nu$-weighted typical cell of the $\beta'$-Delaunay tessellation coincides in distribution with the randomly rescaled $(\nu+1)$-volume-weighted $\beta'$-simplex $\conv(Y'_1,\ldots,Y'_d)$.

Exact formulas for the constants needed to normalize the density of $(Y_1,\ldots,Y_d)$ appearing in (b) and (b$^\prime$) will be given in~\eqref{eq:betamoments} and~\eqref{eq:betaprimemoments}.
\end{remark}

\begin{remark}\label{rem:rep_typical_beta_-1}
Let us point out that in the limiting case $\beta\to -1$ the beta distribution with density $c_{d-1,\beta}(1-\|x\|^2)^{\beta}{\bf 1}_{\BB^{d-1}}(x)$ weakly converges to the uniform distribution on the unit sphere $\SS^{d-2}$, denoted by $\sigma_{d-2}$. Thus, $\PP_{\beta,\nu}$ for fixed $\nu\ge-1$ and $\gamma>0$ weakly converges to a probability measure $\PP_{-1,\nu}$ with
\begin{align*}
\PP_{-1,\nu}(A)
&=
\alpha^{*}_{d,\nu}
\int_{(\SS^{d-2})^d} \sigma_{d-2}(\dd u_1)\ldots \sigma_{d-2}(\dd u_d) \int_{0}^{\infty}\dd r\,{\bf 1}_A(\conv(ru_1,\ldots,ru_d))
\\
&\hspace{4cm}\times r^{d^2-2d+\nu(d-1)}e^{-{2\gamma\kappa_{d-1}\over \omega_d} r^{d-1}} (\Delta_{d-1}(u_1,\ldots,u_d))^{\nu+1},
\end{align*}
where
\[
\alpha^{*}_{d,\nu}=(2\gamma \omega_d^{-1})^{d+\nu-1}{(d-1)(d-1)!^{\nu+1}\pi^{(\nu-1)(d-1)\over 2}\over 2^d \Gamma(d+\nu-1)}{\Gamma({(d+\nu-1)(d-1)\over 2})\over \Gamma({d(d+\nu-2)\over 2}+1)}{\Gamma({d+\nu\over 2})^d\over \Gamma({d+1\over 2})^{d+\nu-1}}\prod\limits_{i=1}^{d-1}{\Gamma({i\over 2})\over \Gamma({i+\nu+1\over 2})}.
\]
This coincides with the formula for the distribution of the $\nu$-weighted typical cell of Poisson-Delaunay tessellation in $\RR^{d-1}$ corresponding to the intensity $2\gamma\omega_d^{-1}$ of underlying Poisson point process, see \cite[Theorem 2.3]{GusakovaThaele} for general $\nu$ {and \cite{Mo89} or \cite[Proposition 4.3.1]{Mo94} for the case $\nu=0$.}
\end{remark}

\begin{remark}
	It should be mentioned that in case of the classical Poisson-Delaunay and Poisson-Voronoi tessellation many more results regarding the distribution of typical cells and faces are available. For example,  in \cite[Lemma 1]{MM95} a precise measure theoretical description of the Palm distribution of the original Poisson point process with respect to the vertex process of the  Poisson-Voronoi tessellation was given. It characterises the distribution of the  Poisson-Delaunay tessellation around its typical cell. A further generalization, providing an explicit description of the Palm distribution of the Poisson-Voronoi tessellation around a uniform random point on the typical $k$-face, was obtained in \cite{BG07}.
	
	In our situation it is also possible to characterise the Palm distribution of $\eta_{\beta}$ and $\eta'_{\beta}$ with respect to the vertex process of the $\beta$- and $\beta'$-Voronoi tessellation, respectively, in the spirit of \cite[Lemma 1]{MM95}. To formulate the corresponding result, let $G=\{g_w:w\in\RR^{d-1}\}$ be the group of translations acting on
	$\RR^{d-1}\times\RR$ by shifts in the spatial ($\RR^{d-1}$-coordinate), that is, $g_w(v,h):=(v+w,h)$ for $(v,h)\in\RR^{d-1}\times\RR$, $w\in\RR^{d-1}$. Next, we denote by
	$$
	\chi_{\beta}^{(')}:=\chi(\eta_{\beta}^{(')}) =\sum\limits_{v\in \cF_0(\cV^{(')}_{\beta})}\delta_v
	$$
	the point process of vertices of the $\beta^{(')}$-Voronoi tessellation $\cV^{(')}_{\beta}$.
	According to the properties of the Poisson point process $\eta^{(')}_{\beta}$ the pair $(\chi_{\beta}^{(')}, \eta^{(')}_{\beta})$ is jointly $G$-stationary. This allows us to define the Palm distribution $\PP_{\beta^{(')}}^{0,\chi}$ of $\eta^{(')}_{\beta}$ with respect to $\chi_{\beta}^{(')}$ as
	$$
	\PP_{\beta^{(')}}^{0,\chi}(A):={1\over \gamma_0}\EE \sum_{v\in\chi_{\beta}^{(')}} {\mathbf 1}(v \in [0,1]^{d-1}, g_{-v}\eta^{(')}_{\beta} \in A), \qquad\qquad A\in\cN(\RR^d),
	$$
	where $\gamma_0:=\gamma_0(\cV^{(')}_{\beta})$ is the intensity of the vertices of $\beta^{(')}$-Voronoi tessellation whose exact value can be derived from the results of Subsection 6.3, see \cite[Section 7.2]{Kallenberg2}.
	
	Given a point set $X$ let us denote by $\Pi^{\downarrow}(X)$ the shift of the standard downward paraboloid $\Pi_{(0,a)}^{\downarrow}$ along the height (time) coordinate, such that $\Pi_{(0,a)}^{\downarrow}$ does not contain any points from $X$ in its interior and that for any $b>a$ we have $\inter\Pi_{(0,b)}^{\downarrow}\cap X\neq \emptyset$. Now, if $\PP_{\beta^{(')}}$ denotes the distribution of the Poisson point process $\eta_{\beta}^{(')}$, then for any measurable function $u:\sfN(\RR^d)\to[0,\infty)$ we have that
	\begin{align*}
		\int_{\sfN(\RR^d)}u(\varphi)\PP_{\beta^{(')}}^{0,\chi}(\dint\varphi)=\int_{\sfN(\RR^d)}\int_{\sfN(\RR^d)}u\big((\varphi_1\cap \Pi^{\downarrow}(\varphi_1))\cup(\varphi_2\cap (\Pi^{\downarrow}(\varphi_1))^{c})\big)\,\PP_{\beta^{(')}}^{0,\chi}(\dint\varphi_1)\PP_{\beta^{(')}}(\dint\varphi_2).
	\end{align*}
	This identity explains how the $\beta^{(')}$-Delaunay tessellation looks like in a neighbourhood of the typical cell. Namely, it is generated by a Poisson point process which is restricted to the complement of the paraboloid defining the typical cell and which is unaffected by the shape of the typical cell. The proof of this formula is a straightforward modification of the proof of Lemma 1 in \cite{MM95} and we do not present the details here in order to keep our presentation short.
\end{remark}

\begin{remark}\label{rem:BetaPolytopes}
	There exists an interesting and very useful connection between the boundary of the convex hull generated by a Poisson point process inside the unit ball, and the boundary of the paraboloid hull process. It was shown in \cite{CSY13} that applying a particular scaling  transformation, the boundary of the convex hull can be mapped into a random surface, which  locally converges to the boundary of the paraboloid hull process, as the intensity of the underlying Poisson point process tends to infinity. This fact appeared to be very useful for investigation of the local properties of the convex hull of Poisson point processes.
	
	In our case, the Poisson point process with intensity measure having density
	$$
	f_{\gamma}(x)=\gamma\,c_{d,\beta}(1-\|x\|^2)^{\beta}{\bf 1}_{\BB^d}(x)
	$$
	generates a convex hull, whose boundary will converge after a suitable transformation locally, as $\gamma\to\infty$, to the boundary of paraboloid hull process $\Phi(\eta_{\beta})$. This is an evidence that there exists a way to build a bridge between $\beta$-polytopes and $\beta$-Delaunay tessellations. In particular, the $\beta$-Delaunay tessellation describes the local limit of $\beta$-polytopes as the number of generating points tends to infinity.
\end{remark}

\begin{proof}[Proof of Theorem~\ref{theo:typical_cell_stoch_rep}]
Let us recall that for any collection of points
$$
x_1:=(v_1,h_1)\in\RR^{d-1}\times\RR\;\; \ldots \;\; x_d:=(v_d,h_d)\in\RR^{d-1}\times\RR
$$
with affinely independent spatial coordinates $v_i$ we denote by $\Pi^{\downarrow}(x_1,\ldots,x_d)$ the unique  translation of the standard downward paraboloid $\Pi^\downarrow$ containing $x_1,\ldots,x_d$ on its boundary. If $x_1,\ldots,x_d$ are distinct points of the Poisson point process $\eta_\beta^{(')}$, then the simplex $K := \conv (v_1,\ldots,v_d)$ belongs to the tessellation $\cD_\beta^{(')}$ if and only if $\inter(\Pi^{\downarrow}(x_1,\ldots,x_d))\cap\eta_\beta^{(')}=\varnothing$, that is if there are no points of $\eta_\beta^{(')}$ strictly inside $\Pi^{\downarrow}(x_1,\ldots,x_d)$. Let us denote the apex of the paraboloid $\Pi^{\downarrow}(x_1,\ldots,x_d)$  by $(w,t)\in \RR^{d-1} \times \RR$. We then have
$$
t - \|v_i-w\|^2 = h_i,\qquad i \in \{1,\ldots,d\}.
$$
As the center of the simplex $\conv(v_1,\ldots,v_d)$ we choose the point $w$ and therefore put
$$
z(x_1,\ldots,x_d):=z(\conv (v_1,\ldots,v_d), \cD(\eta_\beta^{(')}))=w.
$$

We are now ready to begin with the essential part of the proof of Theorem~\ref{theo:typical_cell_stoch_rep}. Fix some Borel set $A\subset \cC'$. From \eqref{eq_2} and the definition of the the generalized centre function for the tessellation $\cD_\beta^{(')}$ we get
\begin{align*}
S_{\eta_\beta^{(')}}(A)
&
:=
\lefteqn{\EE\sum_{(v,M)\in \mu_{\eta_\beta^{(')}}}{\bf 1}_A(M)\,{\bf 1}_{[0,1]^{d-1}}(v)\,(\Vol(M))^{\nu}}\\
&=
{1\over d!}\,
\EE\sum\limits_{(x_1,\ldots,x_d)\in (\eta_\beta^{(')})_{\neq}^d}{\bf 1}_A(\conv(v_1,\ldots,v_d)-z(x_1,\ldots,x_d))\\
&\qquad\times {\bf 1}_{[0,1]^{d-1}}(z(x_1,\ldots,x_d)){\bf 1}\left\{\inter (\Pi^{\downarrow}(x_1,\ldots,x_d))\cap\eta_\beta^{(')}=\varnothing\right\}\\
&\qquad\times \Delta_{d-1}(v_1,\ldots,v_d)^{\nu}.
\end{align*}
Here, $(\eta_\beta^{(')})_{\neq}^d$ denotes the collection of all tuples of the form $(x_1,\ldots,x_d)$ consisting of pairwise distinct points $x_1,\ldots,x_d$  of the Poisson point process $\eta_\beta^{(')}$. We write $S_{\beta}(A):=S_{\eta_{\beta}}(A)$ and $S_{\beta}^{\prime}(A):=S_{\eta_{\beta}^{\prime}}(A)$. Note that $S_{\beta}^{(\prime)}(A)$ is in fact the same as $\lambda_{\beta,\nu}^{(\prime)}\,\PP_{\beta,\nu}^{(\prime)}(A)$, but since at the present moment we don't know whether the normalizing constants $\lambda_{\beta,\nu}$ and $\lambda_{\beta,\nu}^{\prime}$, given by~\eqref{eq:def_gamma_const}, are finite, we prefer to use the notation $S_{\beta}^{(\prime)}(A)$.
Applying the multivariate Mecke formula \cite[Corollary 3.2.3]{SW} to the Poisson point process $\eta_{\beta}^{(\prime)}$ and taking into account \eqref{eq:BetaPoissonIntensity} and \eqref{eq:BetaPrimePoissonIntensity} we obtain
\begin{align}
S_{\beta}^{(\prime)}(A)
&=
{\gamma^d (c^{(\prime)}_{d,\beta})^d\over d!}
\int_{\RR^{d-1}} \dd v_1 \, \ldots  \, \int_{\RR^{d-1}} \dd v_d\,
\int_{0}^{\infty}\dd h_1\, \ldots \, \int_{0}^{\infty} \dd h_d \,
\notag\\
&\qquad\phantom{\times} {\bf 1}_A(\conv(v_1,\ldots,v_d)-z(\tilde x_1,\ldots,\tilde x_d))\notag\\
&\qquad\times {\bf 1}_{[0,1]^{d-1}}(z(\tilde x_1,\ldots,\tilde x_d))\PP\left(\inter (\Pi^{\downarrow}(\tilde x_1,\ldots,\tilde x_d))\cap \eta_{\beta}^{(\prime)}=\varnothing\right)\notag\\
&\qquad\times h_1^{\kappa\beta}\ldots h_d^{\kappa\beta}\,\Delta_{d-1}(v_1,\ldots,v_d)^{\nu}, \label{eq_4}
\end{align}
where $\tilde x_i = (v_i, \kappa h_i)$ for $i=1,\ldots,d$. In the above integral, we are going to pass from the integration over the variables $(v_1,\ldots,v_d,h_1,\ldots,h_d)\in (\RR^{d-1})^d \times(\RR_{+}^*)^d$, where $\RR_+^*=(0,\infty)$, to the integration over certain new variables $(w,r,y_1,\ldots,y_d)\in \RR^{d-1}\times\RR_{+}^*\times\left(\RR^{d-1}\right)^d$ introduced as follows.
Take some tuple $(v_1,\ldots,v_d,h_1,\ldots,h_d)\in (\RR^{d-1})^d \times(\RR_{+}^*)^d$ and assume that $v_1,\ldots,v_d$ are affinely independent.  Denote the apex of the unique downward paraboloid $\Pi^\downarrow(\tilde x_1,\ldots,\tilde x_d)$ whose boundary passes through the points $(v_1,\kappa h_1),\ldots,(v_d,\kappa h_d)$ by $(w,\kappa r^2)\in \RR^{d-1} \times \RR$ and note that  $w= z(\tilde x_1,\ldots,\tilde x_d)$. Observe that in the $\beta'$-case the second coordinate of the apex can be positive, but since such downward paraboloid contains infinitely many points of the Poisson point process $\eta_\beta^{\prime}$ (because any point $(w,0)$ with vanishing second coordinate is an accumulation point for the atoms of $\eta_{\beta}^{\prime}$), we can ignore this possibility in the following.
We can write $v_i = w + ry_i$ with some uniquely defined and pairwise distinct $y_1,\ldots,y_d\in \RR^{d-1}$. The condition that the boundary of the paraboloid passes through the point $(v_i,\kappa h_i)$ reads as $\kappa r^2 - \|v_i-w\|^2 = \kappa h_i$, or $h_i = r^2 (1-\kappa\|y_i\|^2)$. In the $\beta$-case it follows that $y_1,\ldots,y_d \in \BB^{d-1}$.
Let us therefore introduce the transformation $T:\RR^{d-1}\times\RR_{+}^*\times\left(\BB^{d-1}\right)^d\rightarrow\left(\RR^{d-1}\times\RR_{+}^*\right)^d$ (in the $\beta$-case) or $T:\RR^{d-1}\times\RR_{+}^*\times\left(\RR^{d-1}\right)^d\rightarrow\left(\RR^{d-1}\times\RR_{+}^*\right)^d$  (in the $\beta'$-case) defined as
$$
(w,r,y_1,\ldots,y_d)\mapsto \left(ry_1+w,r^2(1-\kappa\|y_1\|^2),\ldots,ry_d+w, r^2(1-\kappa\|y_d\|^2)\right) = (v_1,h_1,\ldots,v_d,h_d).
$$
This transformation is bijective (up to sets of Lebesgue measure zero and provided that in the $\beta$'-case we agree to exclude from the image set all combinations $(v_1,h_1,\ldots,v_d,h_d)$ which lead to a paraboloid whose apex has positive height).  The absolute value of the Jacobian of $T$ is the absolute value of the determinant of the block matrix
$$
J(T):=\left|
\begin{matrix}
E_{d-1} & y_1 & rE_{d-1} & 0 & \dots & 0\\
0 & 2r(1-\kappa\|y_1\|^2) & -2r^2\kappa y_1^{\top} & 0 &\dots & 0\\
E_{d-1} & y_2 & 0 & rE_{d-1} & \dots & 0\\
0 & 2r(1-\kappa\|y_2\|^2) & 0 & -2r^2\kappa y_2^{\top} & \dots & 0\\
\vdots & \vdots &  \vdots & \vdots &\ddots & \vdots \\
E_{d-1} & y_d & 0 & 0 & \dots & rE_{d-1}\\
0 & 2r(1-\kappa\|y_d\|^2) & 0 & 0 & \dots & -2r^2\kappa y_d^{\top}\\
\end{matrix}
\right|,
$$
where $E_{k}$ is the $k\times k$ unit matrix, vectors $y_1,\ldots,y_d$ are considered to be column vectors and $|M|$ stands for the absolute value of the determinant of the matrix $M$. We can compute $J(T)$ as follows:
\begin{align*}
{J(T)\over2^dr^{d^2}} &=\left|
\begin{matrix}
E_{d-1} & y_1 & E_{d-1} & 0 & \dots & 0\\
0 & 1-\kappa\|y_1\|^2 & -\kappa y_1^{\top} & 0 &\dots & 0\\
E_{d-1} & y_2 & 0 & E_{d-1} & \dots & 0\\
0 & 1-\kappa\|y_2\|^2 & 0 & -\kappa y_2^{\top} & \dots & 0\\
\vdots & \vdots & \vdots & \vdots &\ddots & \vdots \\
E_{d-1} & y_d & 0 & 0 & \dots & E_{d-1}\\
0 & 1-\kappa\|y_d\|^2 & 0 & 0 & \dots & -\kappa y_d^{\top}\\
\end{matrix}
\right|=\left|
\begin{matrix}
E_{d-1} & 0 & E_{d-1} & 0 & \dots & 0\\
0 & 1 & -\kappa y_1^{\top} & 0 &\dots & 0\\
E_{d-1} & 0 & 0 & E_{d-1} & \dots & 0\\
0 & 1 & 0 & -\kappa y_2^{\top} & \dots & 0\\
\vdots & \vdots & \vdots & \vdots & \ddots & \vdots \\
E_{d-1} & 0 & 0 & 0 & \dots & E_{d-1}\\
0 & 1 & 0 & 0 & \dots & -\kappa y_d^{\top}\\
\end{matrix}\right|\\
&=\left|
\begin{matrix}
0 & 0 & E_{d-1} & 0 & \dots & 0\\
\kappa y_1^{\top} & 1 & -\kappa y_1^{\top} & 0 &\dots & 0\\
0 & 0 & 0 & E_{d-1} & \dots & 0\\
\kappa y_2^{\top} & 1 & 0 & -\kappa y_2^{\top} & \dots & 0\\
\vdots & \vdots & \vdots & \vdots & \ddots & \vdots \\
0 & 0 & 0 & 0 & \dots & E_{d-1}\\
\kappa y_d^{\top} & 1 & 0 & 0 & \dots & -\kappa y_d^{\top}\\
\end{matrix}\right|
=|\kappa^d|\,\left|
\begin{matrix}
0 & E_{d(d-1)}\\
\begin{matrix}
y_1^{\top} & 1 \\
\vdots & \vdots\\
y_d^{\top} & 1\\
\end{matrix} & \mbox{\normalfont\Large\bfseries 0} \\
\end{matrix}\right|=\left|
\begin{matrix}
1 & \ldots & 1\\
y_1 &\ldots & y_d \\
\end{matrix}\right|
\end{align*}
Thus, $J(T)=2^dr^{d^2}(d-1)! \Delta_{d-1}(y_1,\ldots,y_d)$.
Applying the transformation $T$ in \eqref{eq_4}  we derive
\begin{align}
S_{\beta}^{(\prime)}(A)
&=
{(2\gamma\,c^{(\prime)}_{d,\beta})^d\over d}
\int_{\RR^{d-1}}\dd y_1\, \ldots\, \int_{\RR^{d-1}}\dd y_d\, \int_{\RR^{d-1}} \dd w\, \int_{0}^{\infty}\dd r\,{\bf 1}_A(\conv(ry_1,\ldots,ry_d)) {\bf 1}_{[0,1]^{d-1}}(w)\notag\\
&\qquad\times\PP\left(\{(v,\kappa h)\in\RR^{d-1}\times \kappa \RR_{+}^*\colon \kappa h<-\|v-w\|^2+\kappa r^2\}\cap \eta_{\beta}^{(\prime)}
=\varnothing\right)\,r^{2\kappa d\beta+d^2+\nu(d-1)}\notag\\
&\qquad\times \Delta_{d-1}(y_1,\ldots,y_d)^{\nu+1}\prod_{i=1}^d(1-\kappa\|y_i\|^2)^{\kappa\beta}{\bf 1}(1-\kappa\|y_i\|^2\ge 0).\label{eq_5}
\end{align}
Using now the stationarity of the Poisson point processes $\eta_\beta$ and $\eta_\beta^{\prime}$ with respect to the $v$-coordinate we  conclude that, for any $w\in\RR^{d-1}$,
\begin{align*}
P^{(\prime)} &:=\PP\left(\{(v,\kappa h)\in\RR^{d-1}\times \kappa \RR_{+}^*\colon \kappa h<-\|v-w\|^2+\kappa r^2\}\cap \eta_{\beta}^{(\prime)}=\varnothing\right)\\
&=\PP\left(\{(v,\kappa h)\in\RR^{d-1}\times \kappa \RR_{+}^*\colon \kappa h<-\|v\|^2+\kappa r^2\}\cap \eta_{\beta}^{(\prime)}=\varnothing\right)\\
&=\exp\Big(-\gamma c^{(\prime)}_{d,\beta}\int_{0}^{\infty}\int_{\RR^{d-1}}{\bf 1}(\kappa h<-\|v\|^2+\kappa r^2)h^{\kappa \beta}\dd v\,\dd h\Big).
\end{align*}

For the further computations we need to distinguish the $\beta$- and the $\beta^{\prime}$-cases. For the $\beta$-model we have
\begin{align*}
P&:=\exp\Big(-\gamma c_{d,\beta}\int_{0}^{r^2}\int_{\{v\colon \|v\|\leq (r^2-h)^{1/2}\}}h^{\beta}\dd v\,\dd h\Big)\\
&=\exp\Big(-\gamma \kappa_{d-1}c_{d,\beta}\int_{0}^{r^2}(r^2-h)^{{d-1\over 2}}h^{\beta}\dd h\Big)\\
&=\exp\Big(-\gamma \kappa_{d-1}c_{d,\beta}r^{d+1+2\beta}\int_{0}^{1}(1-h')^{{d-1\over 2}}h'^{\beta}\dd h'\Big)\\
&=\exp\left(-m_{d,\beta} r^{d+1+2\beta}\right),
\end{align*}
where $m_{d,\beta}={\gamma c_{d,\beta}\over \pi c_{d+1,\beta}}$. For the $\beta^{\prime}$-model we obtain
\begin{align*}
P^{\prime}&:=\exp\Big(-\gamma c^{\prime}_{d,\beta}\int_{r^2}^{\infty}\int_{\{v\colon \|v\|\leq (h-r^2)^{1/2}\}}h^{-\beta}\dd v\,\dd h\Big)\\
&=\exp\Big(-\gamma \kappa_{d-1}c^{\prime}_{d,\beta}\int_{r^2}^{\infty}(h-r^2)^{{d-1\over 2}}h^{-\beta}\dd h\Big)\\
&=\exp\Big(-\gamma \kappa_{d-1}c^{\prime}_{d,\beta}r^{d+1-2\beta}\int_{0}^{1}(1-h')^{{d-1\over 2}}h'^{\beta-(d+1)/2-1}\dd h'\Big)\\
&=\exp\left(-m^{\prime}_{d,\beta} r^{d+1-2\beta}\right),
\end{align*}
with $m^{\prime}_{d,\beta}={\gamma c^{\prime}_{d,\beta}\over \pi c^{\prime}_{d+1,\beta}}$. Substituting this into \eqref{eq_5} leads to
\begin{multline}\label{eq_7}
S_{\beta}^{(\prime)}(A)
={(2\gamma\,c^{(\prime)}_{d,\beta})^d\over d}\int_{\RR^{d-1}}\dd y_1\, \ldots\, \int_{\RR^{d-1}} \dd y_d\, \int_{0}^{\infty} \dd r\,{\bf 1}_A(\conv(ry_1,\ldots,ry_d))
\\
\qquad \times r^{2\kappa d\beta+d^2+\nu(d-1)}e^{-m^{(\prime)}_{d,\beta} r^{d+1+2\kappa \beta}}
\Delta_{d-1}(y_1,\ldots,y_d)^{\nu+1}\prod_{i=1}^d(1-\kappa\|y_i\|^2)^{\kappa\beta}{\bf 1}(1-\kappa\|y_i\|^2\ge 0).
\end{multline}
We are now in position to determine the  normalizing constants $\lambda_{\beta,\nu}$ and $\lambda_{\beta,\nu}^{\prime}$ from~\eqref{eq:def_gamma_const}. To this end, we plug $A=\cC'$ (the set of non-empty compact subsets of $\RR^{d-1}$) into the expression~\eqref{eq_7} for $S_\beta^{(\prime)}$. Doing this and using the substitution $s=m^{(\prime)}_{d,\beta} r^{d+1+2\kappa \beta}$, we obtain
\begin{align*}
\lambda_{\beta,\nu}^{(\prime)}
&=
S_{\beta}^{(\prime)}(\cC')\\
&=
{(2\gamma\,c^{(\prime)}_{d,\beta})^d\over d}\int_{(\RR^{d-1})^d}\int_{0}^{\infty} r^{2\kappa d\beta+d^2+\nu(d-1)}e^{-m^{(\prime)}_{d,\beta} r^{d+1+2\kappa \beta}}\\
&\qquad\times \Delta_{d-1}(y_1,\ldots,y_d)^{\nu+1}\prod_{i=1}^d(1-\kappa\|y_i\|^2)^{\kappa\beta}{\bf 1}(1-\kappa\|y_i\|^2\ge 0)\dd r\,\dd y_1\ldots \dd y_d\\
&={(2\gamma\,c^{(\prime)}_{d,\beta})^d\over d(d+1+2\kappa \beta)}(m^{(\prime)}_{d,\beta})^{-d-{(\nu-1)(d-1)\over d+2\kappa \beta+1}}\int_{0}^{\infty} s^{d+{(\nu-1)(d-1)\over d+2\kappa\beta+1}-1}e^{-s}\dd s\\
&\qquad\times\int_{(\RR^{d-1})^d}\Delta_{d-1}(y_1,\ldots,y_d)^{\nu+1}\prod_{i=1}^d(1-\kappa\|y_i\|^2)^{\kappa\beta}{\bf 1}(1-\kappa\|y_i\|^2\ge 0)\dd y_1\ldots \dd y_d\\
&={(2\gamma\,c^{(\prime)}_{d,\beta})^d\Gamma(d+{(\nu-1)(d-1)\over d+2\kappa\beta+1})\over d(d+1+2\kappa\beta)}(m^{(\prime)}_{d,\beta})^{-d-{(\nu-1)(d-1)\over d+2\kappa \beta+1}}\\
&\qquad\times\int_{(\RR^{d-1})^d}\Delta_{d-1}(y_1,\ldots,y_d)^{\nu+1}\prod_{i=1}^d(1-\kappa\|y_i\|^2)^{\kappa\beta}{\bf 1}(1-\kappa\|y_i\|^2\ge 0)\dd y_1\ldots \dd y_d.
\end{align*}

The last integral (which is finite for $\nu\ge-1$) is equal -- up to a constant -- to the $(\nu+1)$-th moment of the volume of random simplex with vertices having a $\beta$- or $\beta^{\prime}$-distribution. The exact values were calculated in \cite[Theorem 2.3]{GKT17} or \cite[Proposition 2.8]{KTT} and are given (in the cases $\kappa=+1$ and $\kappa=-1$) as follows:
\begin{align}
\int_{(\BB^{d-1})^d}&\Delta_{d-1}(y_1,\ldots,y_d)^{\nu+1}\prod\limits_{i=1}^d(1-\|y_i\|^2)^{\beta}\dd y_1\ldots \dd y_d \notag\\
&={1\over (d-1)!^{\nu+1}c_{d-1,\beta}^d}{\Gamma({d+1\over 2}+\beta)^d\over \Gamma({d+\nu\over 2}+\beta+1)^d}{\Gamma({d(d+\nu+2\beta)\over 2} +1)\over \Gamma({d(d+\nu+2\beta)-\nu +1 \over 2})}
\prod\limits_{i=1}^{d-1} {\Gamma({i+\nu+1\over 2})\over \Gamma({i\over 2})},
\label{eq:betamoments}\\
\int_{(\RR^{d-1})^d}&\Delta_{d-1}(y_1,\ldots,y_d)^{\nu+1}\prod\limits_{i=1}^d(1+\|y_i\|^2)^{-\beta}\dd y_1\ldots \dd y_d \notag\\
&={1\over (d-1)!^{\nu+1}(c^{\prime}_{d-1,\beta})^d}{\Gamma({d(2\beta-d-\nu)+\nu+1\over 2})\over\Gamma({d(2\beta-d-\nu)\over 2})}{\Gamma(\beta-{d+\nu\over 2})^d\over\Gamma(\beta-{d-1\over 2})^d}\prod_{i=1}^{d-1}{\Gamma({i+\nu+1\over 2})\over\Gamma({i\over 2})}.\label{eq:betaprimemoments}
\end{align}
We have thus shown that
\begin{align*}
\lambda_{\beta,\nu}&={\Gamma(d+{(\nu-1)(d-1)\over d+2\beta+1})m_{d,\beta}^{-d-{(\nu-1)(d-1)\over d+2 \beta+1}}\over d(d+1+2\beta)(d-1)!^{\nu+1}}\Big({2\gamma\,c_{d,\beta}\Gamma({d+1\over 2}+\beta)\over c_{d-1,\beta}\Gamma({d+\nu\over 2}+\beta+1)}\Big)^d{\Gamma({d(d+\nu+2\beta)\over 2} +1)\over \Gamma({d(d+\nu+2\beta)-\nu +1 \over 2})}\prod\limits_{i=1}^{d-1}{\Gamma({i+\nu+1\over 2})\over \Gamma({i\over 2})},\\
\lambda^{\prime}_{\beta,\nu}&={\Gamma(d+{(\nu-1)(d-1)\over d-2\beta+1})(m^{\prime}_{d,\beta})^{-d-{(\nu-1)(d-1)\over d-2 \beta+1}}\over d(d+1-2\beta)(d-1)!^{\nu+1}}\Big({2\gamma\,c^{\prime}_{d,\beta}\Gamma(\beta-{d+\nu\over 2})\over c^{\prime}_{d-1,\beta}\Gamma(\beta-{d-1\over 2})}\Big)^d{\Gamma({d(2\beta-d-\nu)+\nu+1\over 2})\over\Gamma({d(2\beta-d\nu\over 2})}
\prod\limits_{i=1}^{d-1}{\Gamma({i+\nu+1\over 2})\over \Gamma({i\over 2})}.
\end{align*}
In particular, this implies that $\lambda_{\beta,\nu}, \lambda^{\prime}_{\beta,\nu}<\infty$ provided $\nu\ge-1$.
From \eqref{eq_7} we conclude
\begin{align*}
\PP^{(\prime)}_{\beta,\nu}(A)
=
\frac{S_{\beta}^{(\prime)}}{\lambda_{\beta,\nu}^{(\prime)}}
&=
\alpha^{(\prime)}_{d,\beta,\nu}\int_{(\RR^{d-1})^d}\dd y_1\ldots \dd y_d \, \int_{0}^{\infty}\dd r\,{\bf 1}_A(\conv(ry_1,\ldots,ry_d)) r^{2\kappa d\beta+d^2+\nu(d-1)}\notag\\
&\qquad\times e^{-m^{(\prime)}_{d,\beta} r^{d+1+2\kappa\beta}}
\Delta_{d-1}(y_1,\ldots,y_d)^{\nu+1}\prod_{i=1}^d(1-\kappa\|y_i\|^2)^{\kappa\beta}{\bf 1}(1-\kappa\|y_i\|^2\ge 0),
\end{align*}
with $\alpha_{d,\beta,\nu}$ and $\alpha^{\prime}_{d,\beta,\nu}$ given by \eqref{eq:Alpha} and \eqref{eq:AlphaPrime} respectively.
This completes the argument.
\end{proof}

\section{The volume of weighted typical cells}\label{sec:Volume}

\subsection{Moment formulas}

In this section we apply Theorem \ref{theo:typical_cell_stoch_rep} to compute all moments of the random variables $\Vol(Z_{\beta,\nu})$ and $\Vol(Z^{\prime}_{\beta,\nu})$. These explicit formulas will be the basis of some of the results in part III of this series of papers. In particular, they generalize the moment formulas in \cite{GusakovaThaele} for weighted typical cells in classical Poisson-Delaunay tessellations.

\begin{theorem}\label{theo:volume}
Let $Z_{\beta,\nu}$ be the $\nu$-weighted typical cell of a $\beta$-Delaunay tessellation with $\beta\geq -1$ and $\nu\ge-1$, and let $Z^{\prime}_{\beta,\nu}$ be the $\nu$-weighted typical cell of the $\beta^{\prime}$-Delaunay tessellation with $\beta\geq (d+1)/2$ and $2\beta - d>\nu\ge-1$. Then, for any $s>-\nu-1$, we have
\begin{align*}
\EE \Vol(Z_{\beta,\nu})^s &= {1\over (d-1)!^s}\Big({ \sqrt{\pi}\Gamma({d+1\over 2}+\beta+1)\over \gamma\Gamma({d\over 2}+\beta+1)}\Big)^{{s(d-1)\over d+2\beta+1}}{\Gamma({d(d+2\beta)+\nu(d-1)+1\over 2})\over\Gamma({d(d+2\beta)+(\nu+s)(d-1)+1\over 2})}{\Gamma({d(d+\nu+s +2\beta)\over 2}+1)\over\Gamma({d(d+\nu +2\beta)\over 2}+1)}\\
&\qquad\times{\Gamma(d+{(\nu+s-1)(d-1)\over d+2\beta+1})\over\Gamma(d+{(\nu-1)(d-1)\over d+2\beta+1})}{\Gamma({d+\nu\over 2}+\beta +1)^d\over\Gamma({d+\nu+s\over 2}+\beta +1)^d}\prod\limits_{i=1}^{d-1}{\Gamma({i+\nu+s+1\over 2})\over \Gamma({i+\nu+1\over 2})},
\end{align*}
and for any $2\beta -d-\nu>s>-\nu-1$ we obtain
\begin{align*}
\EE \Vol(Z^{\prime}_{\beta,\nu})^s &= {1\over (d-1)!^s}\Big({ \sqrt{\pi}\Gamma(\beta -{d\over 2})\over \gamma\Gamma(\beta-{d+1\over 2})}\Big)^{{s(d-1)\over d-2\beta+1}}{\Gamma({d(2\beta-d)-(\nu+s)(d-1)+1\over 2})\over\Gamma({d(2\beta-d)-\nu(d-1)+1\over 2})}{\Gamma({d(2\beta-d-\nu)\over 2})\over\Gamma({d(2\beta-d-\nu-s)\over 2})}\\
&\qquad\times{\Gamma(d+{(\nu+s-1)(d-1)\over d-2\beta+1})\over\Gamma(d+{(\nu-1)(d-1)\over d-2\beta+1})}{\Gamma(\beta -{d+\nu+s\over 2})^d\over\Gamma(\beta -{d+\nu\over 2})^d}\prod\limits_{i=1}^{d-1}{\Gamma({i+\nu+s+1\over 2})\over \Gamma({i+\nu+1\over 2})}.
\end{align*}
\end{theorem}
\begin{proof}
Applying Theorem \ref{theo:typical_cell_stoch_rep} we get
\begin{align*}
\EE \Vol(Z^{(\prime)}_{\beta,\nu})^s &= \alpha_{d,\beta,\nu}^{(\prime)}\int_{(\RR^{d-1})^d}\dd y_1\ldots \dd y_d \, \int_{0}^{\infty}\dd r\Vol(\conv(ry_1,\ldots,ry_d))^s r^{2\kappa d\beta+d^2+\nu(d-1)}\notag\\
&\qquad\times e^{-m^{(\prime)}_{d,\beta} r^{d+1+2\kappa\beta}}
\Delta_{d-1}(y_1,\ldots,y_d)^{\nu+1}\prod_{i=1}^d(1-\kappa\|y_i\|^2)^{\kappa\beta}{\bf 1}(1-\kappa\|y_i\|^2\ge 0).
\end{align*}
Then from Fubini's theorem we obtain
\begin{align*}
\EE \Vol(Z^{(\prime)}_{\beta,\nu})^s &= \alpha_{d,\beta,\nu}^{(\prime)}\int_{0}^{\infty}r^{2\kappa d\beta+d^2+\nu(d-1)+s(d-1)}e^{-m^{(\prime)}_{d,\beta} r^{d+1+2\kappa\beta}}\,\dd r\\
&\qquad\times \int_{(\RR^{d-1})^d}\Delta_{d-1}(y_1,\ldots,y_d)^{\nu+1+s}\prod_{i=1}^d(1-\kappa\|y_i\|^2)^{\kappa\beta}{\bf 1}(1-\kappa\|y_i\|^2\ge 0)\,\dd y_1\ldots \dd y_d\\
&=\alpha_{d,\beta,\nu}^{(\prime)}{\Gamma(d+{(\nu-1)(d-1)\over d+2\kappa\beta+1}+{s(d-1)\over d+2\kappa\beta+1})\over (d+1+2\kappa\beta)}(m_{d,\beta}^{(\prime)})^{-d-{(\nu-1)(d-1)\over d+2\kappa\beta+1}-{s(d-1)\over d+2\kappa\beta+1}}\\
&\qquad\times \int_{(\RR^{d-1})^d}\Delta_{d-1}(y_1,\ldots,y_d)^{\nu+1+s}\prod_{i=1}^d(1-\kappa\|y_i\|^2)^{\kappa\beta}{\bf 1}(1-\kappa\|y_i\|^2\ge 0)\,\dd y_1\ldots \dd y_d.
\end{align*}
Finally, using \eqref{eq:betamoments}, \eqref{eq:betaprimemoments} and definition of constants $\alpha_{\beta,d,\nu}$, $\alpha^{\prime}_{\beta,d,\nu}$, $m_{d,\beta}$ and $m^{\prime}_{d,\beta}$ we complete the proof.
\end{proof}

\subsection{Probabilistic representations}

Based on the formulas for the moments of the volume of the random simplices $Z_{\beta, \nu}$ and $Z_{\beta, \nu}'$ we can obtain probabilistic representations for the random variables $\Vol(Z_{\beta,\nu})^2$ and $\Vol(Z_{\beta,\nu}')^2$, which are similar in spirit to the ones for Gaussian or beta random simplices \cite{GKT17,GusakovaThaele}.

Let us first recall some standard distributions. A random variable has a Gamma distribution with shape $\alpha\in(0,\infty)$ and rate $\lambda \in(0,\infty)$ if its density function is given by
\[
g_{\alpha,\lambda}(t)={\lambda^{\alpha}\over\Gamma(\alpha)}t^{\alpha-1}e^{-\lambda t},\quad t\in(0,\infty).
\]
A random variable has a Beta distribution with shape parameters $\alpha_1, \alpha_2\in(0,\infty)$ if its density function is given by
\[
g_{\alpha_1,\alpha_2}(t)={\Gamma(\alpha_1+\alpha_2)\over\Gamma(\alpha_1)\Gamma(\alpha_2)}t^{\alpha_1-1}(1-t)^{\alpha_2 - 1},\quad t\in(0,1).
\]
A random variable has a Beta prime distribution with shape parameters $\alpha_1, \alpha_2\in(0,\infty)$ if its density function is given by
\[
g_{\alpha_1,\alpha_2}(t)={\Gamma(\alpha_1+\alpha_2)\over\Gamma(\alpha_1)\Gamma(\alpha_2)}t^{\alpha_1-1}(1+t)^{-\alpha_1-\alpha_2},\quad t>0.
\]

We will use the notation $\xi\sim \Gam(\alpha,\lambda)$, $\xi\sim \Bet(\alpha,\beta)$ and $\xi\sim \Bet^{\prime}(\alpha,\beta)$ to indicate that random variable $\xi$ has a Gamma distribution with shape $\alpha$ and rate $\lambda$, a Beta distribution with shape parameters $\alpha_1, \alpha_2$, or a Beta prime distribution with shape parameters $\alpha_1, \alpha_2$, respectively. Moreover, $\xi\overset{D}{=}\xi'$ will indicate that two random variables $\xi$ and $\xi'$ have the same distribution.

\begin{theorem}\label{thm:ProbabilisticRepresentation}
The following assertions hold.
\begin{itemize}
\item[(a)] For $\beta\geq -1, \nu\ge-1, d\geq 2$ one has that
\begin{align}
\xi(1-\xi)^{d-1}\left((d-1)!\Vol(Z_{\beta,\nu})\right)^2&\overset{D}{=} (m_{\beta,d}^{-1}\,\rho)^{{2(d-1)\over d+2\beta+1}}(1-\eta)^{d-1}\prod\limits_{i=1}^{d-1}\xi_i,\label{eq:probBeta}
\end{align}
\item[(b)] and for $\beta\geq (d+1)/2,\, 2\beta - d>\nu\ge -1, d\geq 2$ one has that
\begin{align}
(1+\eta^{\prime})^{d-1}\left((d-1)!\Vol(Z^{\prime}_{\beta,\nu})\right)^2&\overset{D}{=} ((m_{\beta,d}^{\prime})^{-1}\,\rho^{\prime})^{{2(d-1)\over d-2\beta+1}}(\xi^{\prime})^{-1}(1+\xi^{\prime})^{d}\prod\limits_{i=1}^{d-1}\xi^{\prime}_i,\label{eq:probBetaPrime}
\end{align}
\end{itemize}
where
\begin{align*}
\xi&\sim\Bet\Big({d+\nu\over 2}+\beta+1, {(d-1)(d+\nu+2\beta)\over 2}\Big),\qquad \xi^{\prime}\sim\Bet^{\prime}\Big(\beta-{d+\nu\over 2}, {(d-1)(2\beta - d -\nu)\over 2}\Big)\\
\eta &\sim\Bet\Big({d+2\beta+1\over 2}, {(d-1)(d+\nu+2\beta)\over 2}\Big),\qquad \eta^{\prime}\sim\Bet^{\prime}\Big(\beta-{d-1\over 2}, {(d-1)(2\beta - d -\nu)\over 2}\Big)\\
\rho&\sim\Gam\Big(d+{(\nu-1)(d-1)\over d+2\beta+1},1\Big),\qquad \rho^{\prime}\sim\Gam\Big(d+{(\nu-1)(d-1)\over d-2\beta+1},1\Big)\\
\xi_i&\sim\Bet\Big({\nu+i+1\over 2}, {d-1-i\over 2}+\beta+1\Big),\qquad \xi^{\prime}_i\sim\Bet^{\prime}\Big({\nu+i+1\over 2}, \beta-{d+\nu\over 2}\Big),\quad i\in\{1,\ldots,d-1\},
\end{align*}
are independent random variables, independent of $\Vol(Z_{\beta,\nu})$ and $\Vol(Z^{\prime}_{\beta,\nu})$, and $m_{\beta,d}$, $m_{\beta,d}^{\prime}$ are defined in \eqref{eq:Constant}.
\end{theorem}

\begin{remark}
Equality \eqref{eq:probBeta} generalizes \cite[Theorem 2.6]{GusakovaThaele} for $\beta=-1$ to general values of $\beta\geq -1$ and \cite[Theorem 2.5 (b)]{GKT17} for $\nu = -1$ to general $\nu \ge -1$. Equality \eqref{eq:probBetaPrime} generalizes \cite[Theorem 2.5 (c)]{GKT17} for $\nu = -1$ to general $2\beta - d>\nu\ge-1$.
\end{remark}

\begin{proof}
First of all let us recall that for $\xi_i$ with $s>-{\nu+1\over 2}$ and for $\xi^{\prime}_i$ with $-{\nu+1\over 2}< s< \beta -{d+\nu\over 2}$ we have
\[
\EE[\xi_i^{s}]={\Gamma({d+\nu\over 2}+\beta+1)\Gamma({i+\nu+1\over 2}+s)\over \Gamma({i+\nu+1\over 2})\Gamma({d+\nu\over 2}+\beta+1+s)},\qquad \EE[(\xi_i^{\prime})^{s}]={\Gamma(\beta-{d+\nu\over 2}-s)\Gamma({i+\nu+1\over 2}+s)\over \Gamma(\beta-{d+\nu\over 2})\Gamma({i+\nu+1\over 2})},
\]
respectively, and for $\rho$ and $\rho^{\prime}$ with $s>0$ we have
\[
\EE\Big[\rho^{{2s(d-1)\over d+2\beta+1}}\Big]={\Gamma(d+{(\nu-1)(d-1)\over d+2\beta+1}+{2s(d-1)\over d+2\beta+1})\over \Gamma(d+{(\nu-1)(d-1)\over d+2\beta+1})},\qquad \EE\Big[(\rho^{\prime})^{{2s(d-1)\over d-2\beta+1}}\Big]={\Gamma(d+{(\nu-1)(d-1)\over d-2\beta+1}+{2s(d-1)\over d-2\beta+1})\over \Gamma(d+{(\nu-1)(d-1)\over d-2\beta+1})},
\]
respectively. Moreover, for $s>-{\nu+1\over 2}$ we compute that
$$
\EE\Big[\xi^{s}(1-\xi)^{(d-1)s}\Big]={\Gamma({(d-1)(d+\nu+2\beta)\over 2}+s(d-1))\Gamma({d+\nu\over 2}+\beta+1+s)\Gamma({d(d+2\beta+\nu)\over 2}+1)\over\Gamma({(d-1)(d+\nu+2\beta)\over 2})\Gamma({d+\nu\over 2}+\beta+1)\Gamma({d(d+2\beta+\nu)\over 2}+1+sd)}
$$
and
$$
\EE\Big[(1-\eta)^{(d-1)s}\Big]={\Gamma({d(d+2\beta)+\nu(d-1)+1\over 2})\Gamma({(d-1)(d+\nu+2\beta)\over 2}+s(d-1))\over\Gamma({d(d+2\beta)+\nu(d-1)+1\over 2}+s(d-1))\Gamma({(d-1)(d+\nu+2\beta)\over 2})}.
$$
Combining this with Theorem \ref{theo:volume} we conclude that, for all $s>-{\nu+1\over 2}$,
$$
(d-1)!^{2s}\,\EE\Big[\xi^{s}(1-\xi)^{(d-1)s}\Vol(Z_{\beta,\nu})^{2s}\Big]=m_{\beta,d}^{-{2s(d-1)\over d+2\beta+1}}\EE\Big[\rho^{{2s(d-1)\over d+2\beta+1}}(1-\eta)^{(d-1)s}\prod\limits_{i=1}^{d-1}\xi_i^s\Big],
$$
which finishes the proof of \eqref{eq:probBeta}.
Analogously for $-{\nu+1\over 2}< s< \beta -{d+\nu\over 2}$ we have
$$
\EE\Big[(\xi^{\prime})^{-s}(1+\xi^{\prime})^{ds}\Big]={\Gamma({(d-1)(2\beta-d-\nu)\over 2}-s(d-1))\Gamma(\beta-{d+\nu\over 2}-s)\Gamma({d(2\beta-d-\nu)\over 2})\over\Gamma({(d-1)(2\beta-d-\nu)\over 2})\Gamma(\beta-{d+\nu\over 2})\Gamma({d(2\beta-d-\nu)\over 2}-sd)}
$$
and
$$
\EE\Big[(1+\eta^{\prime})^{(d-1)s}\Big]={\Gamma({d(2\beta-d)-\nu(d-1)+1\over 2})\Gamma({(d-1)(2\beta-d-\nu)\over 2}-s(d-1))\over\Gamma({d(2\beta-d)-\nu(d-1)+1\over 2}+s(d-1))\Gamma({(d-1)(2\beta-d-\nu)\over 2})}.
$$
By Theorem \ref{theo:volume} for all $-{\nu+1\over 2}< s< \beta -{d+\nu\over 2}$ we obtain  that
$$
(d-1)!^{2s}\,\EE\Big[(1+\eta^{\prime})^{(d-1)s}\Vol(Z_{\beta,\nu}^{\prime})^{2s}\Big]=(m_{\beta,d}^{\prime})^{-{2s(d-1)\over d-2\beta+1}}\EE\Big[\rho^{{2s(d-1)\over d-2\beta+1}}(\xi^{\prime})^{-s}(1+\xi^{\prime})^{ds}\prod\limits_{i=1}^{d-1}(\xi_i^{\prime})^s\Big],
$$
and \eqref{eq:probBetaPrime} follows.
\end{proof}

The next result specifies, for integer values of $\nu$, the connection between the distributions of $\Vol(Z_{\beta,\nu})$ and $\Vol(Z^{\prime}_{\beta,\nu})$ with those of the volume of a beta-simplex and the volume of a beta-prime-simplex as studied in \cite{GKT17}, respectively.

\begin{proposition}
\begin{itemize}
\item[(a)] For $d\geq 2$, $\beta\geq -1$ and integers $\nu\geq -1$ we have
\[
(1-\xi)^{d-1}\,\Vol(Z_{\beta,\nu})^2\overset{D}{=} (m_{\beta,d}^{-1}\,\rho)^{{2(d-1)\over d+2\beta+1}}\Vol\left(\conv(X_0,\ldots, X_{d-1})\right)^2,
\]
where $X_0,\ldots, X_{d-1}$ are independent and identically distributed random points in $\BB^{d+\nu}$ whose distribution has density
$$
c_{d+\nu,\beta}(1-\|x\|)^{\beta},\qquad x\in\BB^{d+\nu},
$$
$\xi\sim\Bet({\nu+1\over 2}, {d(d+2\beta)+\nu(d-1)+1\over 2})$ is independent of $\Vol(Z_{\beta,\nu})$ and $\rho\sim\Gam(d+{(\nu-1)(d-1)\over d+2\beta+1},1)$ is independent of $X_0,\ldots, X_{d-1}$.
\item[(b)] For $d\geq 2$, $\beta\geq (d+1)/2$ and integers $2\beta - d>\nu\ge -1$ we have
\[
(1+\xi^{\prime})^{d-1}\,\Vol(Z^{\prime}_{\beta,\nu})^2\overset{D}{=} ((m_{\beta,d}^{\prime})^{-1}\,\rho^{\prime})^{{2(d-1)\over d-2\beta+1}}\Vol\left(\conv(X^{\prime}_0,\ldots, X^{\prime}_{d-1})\right)^2,
\]
where $X^{\prime}_0,\ldots, X^{\prime}_{d-1}$ are independent and identically distributed random points in $\RR^{d+\nu}$ whose distribution has density
$$
c^{\prime}_{d+\nu,\beta}(1+\|x\|)^{-\beta},\qquad x\in\RR^{d+\nu},
$$
$\xi^{\prime}\sim\Bet^{\prime}({\nu+1\over 2}, {d(2\beta - d -\nu)\over 2})$ is independent of $\Vol(Z^{\prime}_{\beta,\nu})$ and $\rho^{\prime}\sim\Gam(d+{(\nu-1)(d-1)\over d-2\beta+1},1)$ is independent of $X^{\prime}_0,\ldots, X^{\prime}_{d-1}$.
\end{itemize}
\end{proposition}
\begin{proof}
From \cite[Theorem 2.3 (b)]{GKT17} we have
\begin{align*}
(d-1)!^{2s}\EE[\Vol\left(\conv(X_0,\ldots, X_{d-1})\right)^{2s}]&={\Gamma({d+2\beta+\nu+2\over 2})^d\over \Gamma({d+2\beta+\nu+2\over 2}+s)^{d}}{\Gamma({d(d+2\beta+\nu)+2\over 2}+ds)\over \Gamma({d(d+\beta+\nu)+2\over 2}+(d-1)s)}\\
&\qquad\qquad\qquad\qquad\times\prod\limits_{i=1}^{d-1}{\Gamma({\nu+1+i\over 2}+s)\over\Gamma({\nu+1+i\over 2})},
\end{align*}
for any $\beta-{d+\nu\over 2}>s>0$. Using the equality
\begin{align*}
\EE[(1-\xi)^{(d-1)s}]&={\Gamma({d(d+2\beta+\nu)\over 2}+1)\Gamma({d(d+2\beta)+\nu(d-1)+1\over 2}+(d-1)s)\over \Gamma({d(d+2\beta+\nu)\over 2}+1+(d-1)s)\,\Gamma({d(d+2\beta)+\nu(d-1)+1\over 2})},
\end{align*}
and combining this with Theorem \ref{theo:volume} we conclude, for all $\beta-{d+\nu\over 2}>s>0$,
\[
\EE\Big[(1-\xi)^{d-1}\,\Vol(Z_{\beta,\nu})^2\Big]=\EE\Big[(m_{\beta,d}^{-1}\,\rho)^{{2(d-1)\over d+2\beta+1}}\Vol\left(\conv(X_0,\ldots, X_{d-1})\right)^2\Big],
\]
and (a) is proven.
Analogously, from \cite[Theorem 2.3 (b) and (c)]{GKT17} we have
\begin{align*}
&(d-1)!^{2s}\EE[\Vol\left(\conv(X^{\prime}_0,\ldots, X^{\prime}_{d-1})\right)^{2s}]\\
&\qquad\qquad={\Gamma(\beta-{d+\nu\over 2}-s)^d\over \Gamma(\beta-{d+\nu\over 2})^{d}}{\Gamma({d(2\beta-d-\nu)\over 2}-(d-1)s)\over \Gamma({d(2\beta-d-\nu)\over 2}-ds)}\prod\limits_{i=1}^{d-1}{\Gamma({\nu+1+i\over 2}+s)\over\Gamma({\nu+1+i\over 2})},
\end{align*}
and using that
$$
\EE[(1+\xi^{\prime})^{(d-1)s}]={\Gamma({d(2\beta-d-\nu)\over 2}-(d-1)s)\Gamma({d(2\beta - d)-\nu(d-1)+1\over 2})\over \Gamma({d(2\beta-d-\nu)\over 2})\,\Gamma({d(2\beta - d)-\nu(d-1)+1\over 2}-s(d-1))},
$$
together with Theorem \ref{theo:volume} we have for all $s>0$,
\[
\EE\Big[(1+\xi^{\prime})^{d-1}\,\Vol(Z^{\prime}_{\beta,\nu})^2\Big]=\EE\Big[((m_{\beta,d}^{\prime})^{-1}\,\rho^{\prime})^{{2(d-1)\over d-2\beta+1}}\Vol\left(\conv(X^{\prime}_0,\ldots, X^{\prime}_{d-1})\right)^2\Big],
\]
which finishes the proof of (b) and theorem.
\end{proof}

\begin{remark}\rm
The formula for $\Vol(Z_{\beta,\nu})$ is the extension of \cite[Proposition 2.8]{GusakovaThaele} to general $\beta\geq -1$.
\end{remark}
\begin{remark}
In \cite[Theorem 2.7]{GKT17} it was shown that the random variable $\xi$ and the random variable $\xi^{\prime}$ in the previous proposition are equal by distribution to the squared distance from the origin to the $(d-1)$-dimensional affine subspace spanned by the random vectors $X_0,\ldots, X_{d-1}$ and by the random vectors $X^{\prime}_0,\ldots, X^{\prime}_{d-1}$, respectively.
\end{remark}

\section{Angle sums and face intensities}\label{sec:AnglesFaceIntensities}

The aim of this section is to compute the intensity of $j$-dimensional faces in the $\beta$-Delaunay tessellations $\cD_\beta$ and $\cD_{\beta}^\prime$, for all $j\in \{0,\ldots,d-1\}$. Intuitively, the face intensities can be understood as follows. In a stationary random tessellation $\mathcal T$, the expected number of $j$-dimensional faces in a large cube of volume $V$ is asymptotically equivalent to $\gamma_j(\mathcal T)V$, as $V\to\infty$, for a certain constant $\gamma_j(\mathcal T)$, called the {intensity} of $j$-dimensional faces of $\mathcal T$. A precise definition, using Palm calculus, will be given below. To evaluate these constants for the tessellations $\cD_\beta$ and $\cD_{\beta}^\prime$, we will first compute the expected angle sums of the volume-power weighted typical cells of $\cD_\beta$ and $\cD_{\beta}^\prime$.

\subsection{Expected angle sums of weighted typical cells}
Let us recall that $Z_{\beta,\nu}$ and $Z_{\beta,\nu}^\prime$ denote the typical cells of $\cD_\beta$ and $\cD_{\beta}^\prime$ weighted by the $\nu$-th power of their volume. Our aim is to compute the expected angle sums of these random simplices. First we need to introduce the necessary notation.

Given a simplex $T := \conv(Z_1,\ldots,Z_d) \subset \RR^{d-1}$, we denote by $\sigma_k(T)$ the \textbf{sum of internal angles} of $T$ at all its $k$-vertex faces of the form $\conv(Z_{i_1},\ldots, Z_{i_k})$, that is
$$
\sigma_k(T) = \sum_{\substack{1\leq i_1<\ldots< i_k\leq d\\F := \conv(Z_{i_1},\ldots,Z_{i_k})}} \beta(F,T), \qquad k\in \{1,\ldots,d\}.
$$
Here, $\beta(F,T)$ is the internal angle of $T$ at its face $F$ normalized in such a way that the angle of the full space is $1$, see \cite[p.\ 458]{SW}.
If $Z_1,\ldots,Z_d$ are $d$ i.i.d.\ random points in $\BB^{d-1}$ distributed according to the beta density
$$
f_{d-1,\beta}(z) = c_{d-1,\beta} (1-\|z\|^2)^{\beta}, \qquad z\in \BB^{d-1},
$$
then $\conv(Z_1,\ldots,Z_d)$ is called the beta simplex with parameter $\beta>-1$. The beta simplex with parameter $\beta=-1$ is defined as $\conv(Z_1,\ldots,Z_d)$, where $Z_1,\ldots,Z_d$ are i.i.d.\ uniform on the unit sphere $\SS^{d-2}$.
The expected angle sums of these simplices, denoted by
$$
\mathbb J_{d,k}(\beta) : = \EE \sigma_k(\conv(Z_1,\ldots,Z_d)), \qquad k\in \{1,\ldots,d\},
$$
have been computed in~\cite{kabluchko_formula}, see Theorem~1.2 and the discussion thereafter. According to this formula, we have
\begin{equation}\label{eq:J_nk_integral}
\mathbb J_{d,k}\left(\frac{\alpha-d+1}{2}\right)
=
\binom dk \int_{-\infty}^{+\infty} c_{\frac{\alpha d}2} (\cosh u)^{-\alpha d - 2}
\left(\frac 12  + \ii \int_0^u  c_{\frac{\alpha-1}{2}} (\cosh v)^{\alpha}\dd v \right)^{d-k} \dd u,
\end{equation}
for all $d\geq 3$, $k\in \{1,\ldots,d\}$ and $\alpha \geq  d-3$, where
\begin{equation*}
c_{\gamma} :=  c_{1,\gamma} = \frac{ \Gamma\left(\gamma + \frac{3}{2} \right) }{  \sqrt \pi\, \Gamma (\gamma+1)}, \qquad \gamma>-1.
\end{equation*}
Similarly, let $Z_1^\prime,\ldots,Z_d^\prime$ be $d$ i.i.d.\ random points in $\RR^{d-1}$ distributed according to the beta$^\prime$ density
$$
f^\prime_{d-1,\beta}(z) = c_{d-1,\beta}^\prime (1+\|z\|^2)^{-\beta}, \qquad z\in \RR^{d-1}.
$$
Then, $\conv(Z_1^\prime,\ldots,Z_d^\prime)$ is called the beta simplex with parameter $\beta> \frac{d-1}{2}$.
The expected angle sums of the beta$^\prime$ simplices, denoted by
$$
\mathbb J_{d,k}^\prime(\beta) : = \EE \sigma_k(\conv(Z_1^\prime,\ldots,Z_d^\prime)), \qquad k\in \{1,\ldots,d\},
$$
have been computed in~\cite{kabluchko_formula}, see Theorem~1.7 and the discussion thereafter:
\begin{equation*}
\mathbb J_{d,k}^\prime\left(\frac{\alpha+d-1}{2}\right)
=
\binom dk \int_{-\infty}^{+\infty} c_{\frac{\alpha d}2}^\prime (\cosh u)^{-(\alpha d - 1)}
\left(\frac 12  + \ii \int_0^u  c_{\frac{\alpha+1}{2}}^\prime (\cosh v)^{\alpha-1}\dd v \right)^{d-k} \dd u,
\end{equation*}
for all $d\in\NN$, $k\in \{1,\ldots,d\}$ and for all $\alpha >0$ such that $\alpha d >1$. Here,
\begin{equation*}
c_{\gamma}^\prime :=  c_{1,\gamma}^\prime = \frac{ \Gamma (\gamma) }{  \sqrt \pi\, \Gamma \left(\gamma - \frac 12\right)}, \qquad \gamma > \frac 12.
\end{equation*}

We are now going to state a formula for the expected angle sums of the typical cells $Z_{\beta,\nu}$ and $Z_{\beta,\nu}^\prime$. Note that we include the case of $Z_{-1,\nu}$ which is interpreted as the $\nu$-weighted typical cell in the classical Poisson-Delaunay tessellation; see Remark~\ref{rem:rep_typical_beta_-1}.
\begin{theorem}\label{theo:angle_sum_cell}
Let $Z_{\beta,\nu}$ be the $\nu$-weighted typical cell of the $\beta$-Delaunay tessellation with $\beta\geq -1$ and integer $\nu\ge -1$.
Also, let $Z_{\beta,\nu}^\prime$ be the $\nu$-weighted typical cell of the $\beta^\prime$-Delaunay tessellation with $\beta > (d+1)/2$ and integer $\nu$ such that  $2\beta-d > \nu \ge-1$.
Then, for all $k\in \{1,\ldots,d\}$,
\begin{align*}
\EE \sigma_k(Z_{\beta,\nu})
&=
\mathbb J_{d,k}\left(\beta + \frac {\nu+1} 2\right)\\
&=
\binom dk \int_{-\infty}^{+\infty} c_{\frac{\alpha d}2} (\cosh u)^{-\alpha d - 2}
\left(\frac 12  + \ii \int_0^u  c_{\frac{\alpha-1}{2}} (\cosh v)^{\alpha}\dd v \right)^{d-k} \dd u,
\\
\EE \sigma_k(Z_{\beta,\nu}^\prime)
&=
\mathbb J_{d,k}^\prime\left(\beta - \frac {\nu+1} 2\right)\\
&=
\binom dk \int_{-\infty}^{+\infty} c_{\frac{\alpha^\prime d}2}^\prime (\cosh u)^{-(\alpha^\prime d - 1)}
\left(\frac 12  + \ii \int_0^u  c_{\frac{\alpha^\prime-1}{2}}^\prime (\cosh v)^{\alpha^\prime-1}\dd v \right)^{d-k} \dd u,
\end{align*}
where $\alpha = 2\beta + \nu + d$ and $\alpha' = 2\beta - \nu - d$.
\end{theorem}
\begin{proof}
We consider the $\beta$-Delaunay case. Let first $\beta>-1$. Since rescaling does not change angle sums, it follows from Remark~\ref{rem:rep_typical} that
$$
\EE \sigma_k(Z_{\beta,\nu}) = \EE \sigma_k(\conv(Y_1,\ldots,Y_d)),
$$
where $(Y_1,\ldots,Y_d)$ are $d$ random points in the unit ball $\BB^{d-1}$ whose joint density is proportional to
$$
\Delta_{d-1}(y_1,\ldots,y_d)^{\nu+1} \, \prod\limits_{i=1}^d(1-\|y_i\|^2)^{\beta},
\qquad
y_1\in \BB^{d-1},\ldots,y_d\in \BB^{d-1}.
$$
On the other hand, let $Y_1^*,\ldots,Y_d^*$ be $d$ i.i.d.\ random points in $\BB^{d-1}$ with joint density proportional to
$$
\prod\limits_{i=1}^d(1-\|y_i^*\|^2)^{\beta+ \frac {\nu+1} 2}.
$$
By Remark 4.2 of~\cite{beta_polytopes}, we have
$$
\EE \sigma_k(\conv(Y_1,\ldots,Y_d)) = \EE \sigma_k(\conv(Y_1^*,\ldots,Y_d^*)).
$$
The expected angle sums of $\conv(Y_1^*,\ldots,Y_d^*)$ are
\begin{equation*}
\EE \sigma_k(\conv(Y_1^*,\ldots,Y_d^*))
=
\mathbb J_{d,k}\left(\beta + \frac {\nu+1} 2\right),
\end{equation*}
which is given by~\eqref{eq:J_nk_integral} with  $\alpha = 2\beta + \nu + d$. Taking everything together, proves the theorem in the $\beta$-Delaunay case with $\beta>-1$.

For the $\beta$-Delaunay case with $\beta=-1$ (where $Z_{-1,\nu}$ is interpreted as the $\nu$-weighted cell in the classical Poisson-Delaunay tessellation), the starting point is Remark~\ref{rem:rep_typical_beta_-1} which implies that
$$
\EE \sigma_k(Z_{-1,\nu}) = \EE \sigma_k(\conv(Y_1,\ldots,Y_d)),
$$
where $(Y_1,\ldots,Y_d)$ are $d$ random points in the unit sphere $\SS^{d-2}$ whose joint probability law is proportional to
$$
\Delta_{d-1}(y_1,\ldots,y_d)^{\nu+1} \sigma_{d-2}(\dd y_1)\ldots \sigma_{d-2}(\dd y_d),
\qquad
y_1\in \SS^{d-2},\ldots, y_d\in \SS^{d-2}.
$$
The rest of the proof is similar to the case $\beta>-1$. The $\beta^\prime$-case is similar as well.
\end{proof}

We stated Theorem~\ref{theo:angle_sum_cell} for integer $\nu$ only because this assumption is required by the method of proof of Remark 4.2 in~\cite{beta_polytopes}. Specifying Theorem~\ref{theo:angle_sum_cell} to $\nu=0$, respectively $\nu=1$, we obtain the expected angle sums of the typical cell, respectively, the cell containing the origin, of the $\beta$-Delaunay tessellation, for $\beta\geq -1$.  For the typical cell ($\nu=0$) of the classical Poisson-Delaunay tessellation in $\RR^{d-1}$ (corresponding to $\beta=-1$), the angle sum is given by
$$
\EE \sigma_k (Z_{-1,0}) = \mathbb J_{d,k}\left(-\frac 12 \right), \qquad k\in \{1,\ldots,d\}.
$$
Applying to the quantity on the right-hand side Theorem~1.3 of~\cite{kabluchko_formula}, we arrive at the following result.

\begin{theorem}\label{theo:angle_sum_cell_beta_-1}
Let $Z=Z_{-1,0}$ be the typical cell of the classical Poisson-Delaunay tessellation $\cD_{-1}$ in $\RR^{d-1}$. Then, for all $k\in \{1,\ldots,d\}$ such that $(d-1)(k-1)$ is even, we have
$$
\EE \sigma_{k}(Z) =
\binom{d}{k} \left(\frac{\Gamma(\frac{d}{2})}{\sqrt{\pi}\, \Gamma(\frac{d-1}{2})}\right)^{d-k} \cdot \frac{\sqrt \pi\, \Gamma(\frac{(d-1)^2+2}{2}) }{\Gamma(\frac{(d-1)^2+1}{2})}
\cdot \Res\limits_{x=0} \left[\frac{\left(\int_{0}^x (\sin y)^{d-2} \dd y\right)^{d-k}}{(\sin x)^{(d-1)^2+1}}\right].
$$
\end{theorem}
In the case when both $d$ and $k$ are even, Proposition~1.4 of~\cite{kabluchko_formula} yields a formula for $\EE \sigma_k(Z)$ which is more complicated than the one given in Theorem~\ref{theo:angle_sum_cell_beta_-1}.

\subsection{Behaviour of $\beta$-Delaunay cells and their expected angle sums as $\beta\to\infty$}

\begin{figure}[t]\label{fig:ConvergenceToRegular}
\centering
\includegraphics[width=0.5\columnwidth]{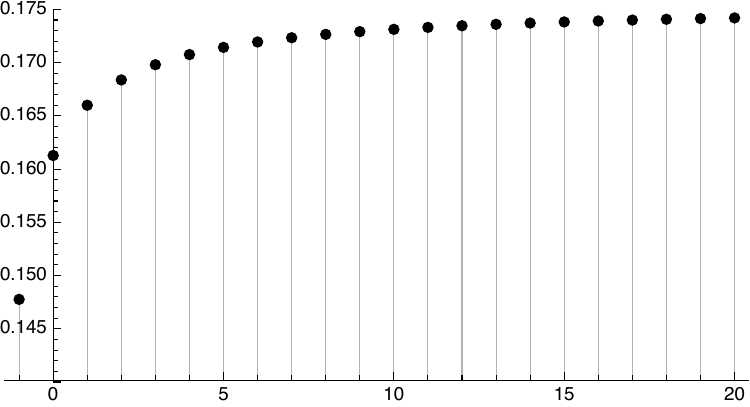}
\caption{Numerical values for the expected angle sums $\EE\sigma_k(Z_{\beta,\nu})$ of the $\nu$-weighted typical Delaunay simplex in $\RR^{d-1}$ with $\beta\in\{-1,0,\ldots,20\}$, $\nu=0$, $d=4$ and $k=1$. The corresponding angle sum of a regular simplex is $\frac 3 \pi \arccos \frac 13 -1\approx 0.1755$.}
\end{figure}

Figure \ref{fig:ConvergenceToRegular} shows the numerical values for the expected angle sums $\EE\sigma_1(Z_{\beta,\nu})$ of the $\nu$-weighted typical $\beta$-Delaunay simplex in $\RR^{d-1}$ for $\beta\in\{-1,0,\ldots,20\}$, with $\nu=0$ and $d=4$. It suggests that, as $\beta$ grows, $\EE\sigma_1(Z_{\beta,\nu})$ approaches the value ${3\over \pi} \arccos \frac 13 -1\approx 0.1755$, which is the angle sum $\sigma_1(\Sigma_3)$ of a regular simplex $\Sigma_3$ in $\RR^3$. The next proposition confirms this conjecture in full generality, that is, for general dimensions $d$, weights $\nu$, and $k\in\{1,\ldots,d\}$. Moreover, it states that the weak limit of $Z_{\beta,\nu}$ as $\beta\to\infty$ is the volume -weighted Gaussian simplex whose angle sums, by coincidence, are the same as for the regular simplex.   We will come back to this behaviour in part II of this series of papers, where we shall describe the limit of the whole $\beta$-Delaunay tessellations as $\beta\to\infty$.

\begin{proposition}
Fix $d\geq 2$ and $\nu>-1$. Then, as $\beta\to\infty$, the distribution of $\sqrt{2\beta}\, Z_{\beta,\nu}$ as well as that of $\sqrt{2\beta}\, Z_{\beta,\nu}^\prime$, converges weakly on the space $\mathcal C'$ to the distribution of the volume-power weighted Gaussian random simplex $\conv(G_1,\ldots, G_d)$, where $(G_1,\ldots,G_d)$ are $d$ random points in $\RR^{d-1}$ with the joint density given by
$$
\frac{(d-1)!^{\nu+1}}{ d^{\frac{\nu+1}2} 2^{(\nu+1)(d-1)/2}}
\left(\prod\limits_{i=1}^{d-1} {\Gamma({i\over 2})\over \Gamma({i+\nu+1\over 2})}\right)
\Delta_{d-1}(g_1,\ldots,g_d)^{\nu+1}
\left(\frac {1}{\sqrt{2\pi}}\right)^{d(d-1)} \prod_{i=1}^d  e^{-\|g_i\|^2/2},
$$
for $g_1\in\RR^{d-1},\ldots,g_d\in\RR^{d-1}$.
Also, if $\nu\ge-1$ is an integer, then for all $k\in\{1,\ldots,d\}$ one has that
$$
\lim_{\beta\to\infty}\EE\sigma_k(Z_{\beta,\nu})
=
\lim_{\beta\to\infty}\EE\sigma_k(Z_{\beta,\nu}^\prime)
=
\sigma_k(\Sigma_{d-1}),
$$
where $\Sigma_{d-1}$ stands for a regular simplex with $d$ vertices in $\RR^{d-1}$.
\end{proposition}
\begin{proof}
For concreteness, we consider the $\beta$-case. From Remark \ref{rem:rep_typical} and Equation~\eqref{eq:betamoments} we know that $Z_{\beta,\nu}$ has the same distribution as the random simplex $\conv(RY_1,\ldots,RY_d)$, where
\begin{itemize}
\item[(a)] $R$ is a random variable with density proportional to $r^{2d\beta+d^2+\nu(d-1)}e^{-m_{d,\beta}r^{d+1+2\beta}}$, $r\in(0,\infty)$,
\item[(b)] $Y_1,\ldots,Y_d$ are random vectors with joint density equal to
\begin{multline*}
f(y_1,\ldots,y_d) := \left(
{1\over (d-1)!^{\nu+1}c_{d-1,\beta}^d}{\Gamma({d+1\over 2}+\beta)^d\over \Gamma({d+\nu\over 2}+\beta+1)^d}{\Gamma({d(d+\nu+2\beta)\over 2} +1)\over \Gamma({d(d+\nu+2\beta)-\nu +1 \over 2})}
\prod\limits_{i=1}^{d-1} {\Gamma({i+\nu+1\over 2})\over \Gamma({i\over 2})}\right)^{-1}
\\
\times\Delta_{d-1}(y_1,\ldots,y_d)^{\nu+1}\prod_{i=1}^d(1-\|y_i\|^2)^\beta,\qquad y_1\in\BB^{d-1},\ldots,y_d\in\BB^{d-1},
\end{multline*}
\item[(c)] $R$ is independent from $(Y_1,\ldots,Y_d)$.
\end{itemize}
Let us first show that $R\to 1$ in probability, as $\beta\to\infty$. Define $\alpha:=d+{(\nu-1)(d-1)\over d+1+2\beta}$, let $Z\sim\Gamma(\alpha,1)$ and observe that $R$ and $(m_{d,\beta}^{-1}Z)^{1/(d+1+2\beta)}$ are identically distributed. We thus compute that
\begin{align*}
\EE[R] &= {1\over m_{d,\beta}^{1/(d+1+2\beta)}}\EE[Z^{1/(d+1+2\beta)}] \\
&=m_{d,\beta}^{-1/(d+1+2\beta)}{1\over\Gamma(\alpha)}\int_0^\infty z^{1/(d+1+2\beta)}z^{\alpha-1}e^{-z}\,\dint z\\
&= \Bigg({2\sqrt{\pi}\Gamma({d\over 2}+\beta+{3\over 2})\over\Gamma({d\over 2}+\beta+1)}\Bigg)^{1\over d+1+2\beta}{\Gamma(\alpha+{1\over d+1+2\beta})\over \Gamma(\alpha)}=1+O(\beta^{-1}\log\beta)
\end{align*}
and, similarly,
\begin{align*}
\VV[R]
&= \Bigg({2\sqrt{\pi}\Gamma({d\over 2}+\beta+{3\over 2})\over\Gamma({d\over 2}+\beta+1)}\Bigg)^{2\over d+1+2\beta}\Bigg({\Gamma(\alpha+{2\over d+1+2\beta})\over \Gamma(\alpha)}-{\Gamma(\alpha+{1\over d+1+2\beta})^2\over \Gamma(\alpha)^2}\Bigg)={\psi^{(1)}(d+1)\over 4\beta^2}+O(\beta^{-3})
\end{align*}
as $\beta\to\infty$, by a multiple application of the asymptotic expansion of the gamma function (here, $\psi^{(1)}$ stands for the first polygamma function, that is, the first derivative of $\ln\Gamma(x)$). As a consequence, using Chebyshev's inequality we have that, for any $\varepsilon>0$,
\begin{align*}
\PP(|R-\EE[R]|>\varepsilon) \leq {\VV[R]\over\varepsilon^2} \to 0,
\end{align*}
as $\beta\to\infty$. In other words, $R-\EE[R]$ converges to zero in probability, as $\beta\to\infty$. Since $\EE[R]\to 1$, we also have that $R$ converges in probability to the constant random variable $1$, as $\beta\to\infty$, by Slutsky's theorem.

Next, we claim that  $\sqrt{2\beta}(Y_1,\ldots,Y_d)$ converges in distribution, as $\beta\to\infty$, to a $d$-tuple $(G_1,\ldots,G_d)$ of random vectors in $\RR^{d-1}$ with certain joint density which we will compute.
Indeed, the density of $\sqrt{2\beta}(Y_1,\ldots,Y_d)$ is given by
$$
\left(\frac 1 {\sqrt{2\beta}}\right)^{d(d-1)}f\left(\frac{y_1}{\sqrt{2\beta}},\ldots,\frac{y_d}{\sqrt{2\beta}}\right)
$$
We now let $\beta\to\infty$. Using that $(1-\|y\|^2/(2\beta))^\beta\to e^{-\|y\|^2/2}$ and the standard asymptotics $\Gamma(\beta + c_1)/\Gamma(\beta+c_2) \sim \beta^{c_1-c_2}$, we obtain that the above density converges pointwise to
$$
\frac{(d-1)!^{\nu+1}}{ d^{\frac{\nu+1}2} 2^{(\nu+1)(d-1)/2}}
\left(\prod\limits_{i=1}^{d-1} {\Gamma({i\over 2})\over \Gamma({i+\nu+1\over 2})}\right)
\Delta_{d-1}(y_1,\ldots,y_d)^{\nu+1}
\left(\frac {1}{\sqrt{2\pi}}\right)^{d(d-1)} \prod_{i=1}^d  e^{-\|y_i\|^2/2}.
$$
By Scheff\'e's lemma, the tuple $\sqrt{2\beta}(Y_1,\ldots,Y_d)$ converges weakly to the tuple $(G_1,\ldots,G_d)$ with the above joint density. A related result without the volume-power weighting can be found in Lemma 1.1 in \cite{beta_polytopes}.

Using now Slutsky's theorem again together with the continuous mapping theorem we conclude that, as $\beta\to\infty$, the random simplex
$$
\sqrt{2\beta}Z_{\beta,\nu}=\sqrt{2\beta}\conv(RY_1,\ldots,RY_d)
$$
converges in distribution (on the space of convex bodies in $\RR^{d-1}$ supplied with the Hausdorff distance) to a weighted Gaussian simplex $\conv(G_1,\ldots,G_d)$; the continuity of the involved map is guaranteed by \cite[Theorem 12.3.5]{SW}. Using once again the continuous mapping theorem, this implies that, as $\beta\to\infty$,
$$
\sigma_k(\sqrt{2\beta}Z_{\beta,\nu}) \overset{d}{\longrightarrow}\sigma_k(\conv(G_1,\ldots,G_d))
$$
for all $k\in\{1,\ldots,d\}$, since angle sums are invariant under rescaling. As they are also bounded, the sequence of random variables $\sigma_k(\sqrt{2\beta}Z_{\beta,\nu})$, $\beta>-1$, is uniformly integrable and we have that
$$
\lim_{\beta\to\infty}\EE[\sigma_k(\sqrt{2\beta}Z_{\beta,\nu})] = \EE[\sigma_k(\conv(G_1,\ldots,G_d))].
$$
However, by taking the limit $\beta\to\infty$ in \cite[Remark 4.2]{beta_polytopes} we have that the expected angle sum of the weighted Gaussian simplex $\conv(G_1,\ldots,G_d)$ coincides with the expected angle sum of a standard (unweighted) Gaussian simplex $\conv(N_1,\ldots,N_d)$, where $N_1,\ldots,N_d$ are are i.i.d.\ standard Gaussian random vectors in $\RR^{d-1}$. Finally, let $\Sigma_{d-1}$ be a regular simplex in $\RR^{d-1}$ and recall from \cite{GoetzeKabluchkoZap,GaussianSimplexAngles} that the expected angle sum $\EE[\sigma_k(\conv(N_1,\ldots,N_d))]$ coincides with $\sigma_k(\Sigma_{d-1})$. We have thus shown that
\begin{align*}
\lim_{\beta\to\infty}\EE\sigma_k(Z_{\beta,\nu}) &=\lim_{\beta\to\infty}\EE[\sigma_k(\sqrt{2\beta}Z_{\beta,\nu})] \\
&= \EE[\sigma_k(\conv(G_1,\ldots,G_d))]\\
&= \EE[\sigma_k(\conv(N_1,\ldots,N_d))]\\
&= \sigma_k(\Sigma_{d-1}),
\end{align*}
and the proof is complete.
\end{proof}

\subsection{Face intensities in $\beta$-tessellations}\label{subsec:intensities}
Given a stationary tessellation $\cT$ on $\RR^{d-1}$, one can introduce the notion of face intensities for faces of all dimensions $j\in\{0,\ldots, d-1\}$ as follows~\cite[p.~450 and \S~4.1]{SW}. Fix some center function $z:\cC' \to \RR^{d-1}$.  Let $\cF_{j}(\cT)$ be the set of all $j$-dimensional faces of the cells of the tessellation $\cT$. By convention, each face is counted once even if it is a face of two or more cells. Consider the point process
$$
\pi_{j}(\cT) := \sum_{F\in \cF_j(\cT)} \delta_{z(F)}
$$
on $\RR^{d-1}$ and note that it is stationary because the center function is required to be translation invariant. The \textbf{intensity} of $j$-dimensional cells of $\cT$ is just the intensity of this point process, that is
$$
\gamma_j(\cT) := \EE \sum_{F\in \cF_j(\cT)} {\bf 1}_{[0,1]^{d-1}}(z(F)).
$$
In the next theorem we compute the cell intensities in the $\beta$- and $\beta^\prime$-Delaunay tessellations $\cD_{\beta}$ and $\cD_\beta^\prime$ on $\RR^{d-1}$.

\begin{theorem}\label{theo:cell_intensities}
For all $j\in\{0,\ldots, d-1\}$ and $\beta\geq -1$ (in the $\beta$-case) or $\beta>(d+1)/2$ (in the $\beta^\prime$-case), we have
$$
\gamma_j(\cD_{\beta}) = \frac {\mathbb J_{d,j+1}\left(\beta + \frac {1} 2\right)} {\EE \Vol (Z_{\beta,0})},
\qquad
\gamma_j(\cD_{\beta}^\prime) = \frac {\mathbb J_{d,j+1}^\prime\left(\beta - \frac {1} 2\right)} {\EE \Vol (Z_{\beta,0}^\prime)},
$$
where $\EE \Vol (Z_{\beta,0}^{(\prime)})$ is as in Theorem~\ref{theo:volume} and
\begin{align*}
\mathbb J_{d,j+1}\left(\beta + \frac {1} 2\right)
&=
\binom d{j+1} \int_{-\infty}^{+\infty} c_{\frac{(2\beta + d) d}2} (\cosh u)^{- (2\beta + d) d - 2}\\
&\hspace{3cm}\times\left(\frac 12  + \ii \int_0^u  c_{\frac{2\beta + d - 1}{2}} (\cosh v)^{2\beta + d}\dd v \right)^{d-j-1} \dd u,\\
\mathbb J_{d,j+1}^\prime\left(\beta - \frac {1} 2\right)
&=
\binom d{j+1} \int_{-\infty}^{+\infty} c_{\frac{(2\beta-d) d}2}^\prime (\cosh u)^{-(2\beta-d) d + 1}\\
&\hspace{3cm}\times
\left(\frac 12  + \ii \int_0^u  c_{\frac{2\beta - d - 1}{2}}^\prime (\cosh v)^{2\beta-d-1}\dd v \right)^{d-j-1} \dd u.
\end{align*}
\end{theorem}
\begin{proof}
For concreteness, we consider the beta case. According to Theorem~10.1.3 of~\cite{SW}, the cell intensities  of $\cD_{\beta}$  satisfy
\begin{equation*}
\gamma_j(\cD_{\beta}) = \gamma_{d-1} (\cD_{\beta}) \cdot \EE \sigma_{j+1}(Z_{\beta,0}), \qquad j\in \{0,\ldots,d-1\},
\end{equation*}
where $Z_{\beta,0}$ is the typical cell of the tessellation $\cD_{\beta}$ (that is, a random simplex distributed according to $\PP_{\beta,0}$). The intensity of the cells of maximal dimension $d-1$ is known to satisfy
$$
\gamma_{d-1}(\cD_{\beta}) = \frac 1 {\EE \Vol (Z_{\beta,0})},
$$
see \cite[Equation (10.4)]{SW}.
On the other hand, by Theorem~\ref{theo:angle_sum_cell},
\begin{align*}
\EE \sigma_{j+1}(Z_{\beta,0})
&=
\mathbb J_{d,j+1}\left(\beta + \frac {1} 2\right),
\end{align*}
and the same theorem yields also an explicit expression for $\mathbb J_{d,j+1}(\beta + \frac {1} 2)$. Taking these three equations together completes the proof in the $\beta$-case. The $\beta^\prime$-case is similar.
\end{proof}

Using duality, we can also compute the face intensities of the $\beta$- and $\beta^\prime$-Voronoi tessellations.
\begin{proposition}\label{prop:duality_face_intensities}
The face intensities of $\cD_{\beta}^{(\prime)}$ and $\cV_{\beta}^{(\prime)}$ are related via
$$
\gamma_{k-1}(\cD_{\beta}^{(\prime)}) = \gamma_{d-k}(\cV_{\beta}^{(\prime)}),
\qquad
k\in \{1,\ldots,d\}.
$$
\end{proposition}
\begin{proof}
Since the $\beta^{(')}$-Voronoi tessellation $\cV_{\beta}^{(\prime)}$ is dual to the $\beta^{(\prime)}$-Delaunay tessellation $\cD_{\beta}^{(\prime)}$, each $(k-1)$-dimensional face of $\cD_\beta^{(\prime)}$ corresponds to a $(d-k)$-dimensional face of $\cV_\beta^{(')}$, and the claim follows.
\end{proof}

\subsection{Expected face numbers of the typical $\beta$-Voronoi cell}
In the next theorem we compute the expected \textbf{$f$-vector} of the typical cell of the $\beta$- and $\beta^\prime$-Voronoi tessellations $\cV_{\beta}$ and $\cV_\beta^\prime$.

\begin{theorem}\label{theo:typical_beta_poi_vor_f_vect}
Let $Y_{\beta}^{(\prime)}$ be the typical cell of the $\beta^{(\prime)}$-Voronoi tessellation $\cV_{\beta}^{(\prime)}$ in $\RR^{d-1}$, where, as usual, $\beta \geq  -1$ in the $\beta$-case and $\beta>(d+1)/2$ in the $\beta^\prime$-case.  Then, for all $k\in \{1,\ldots,d\}$,
$$
\EE f_{d-k}(Y_{\beta}) =  k \gamma_{k-1} (\cD_\beta) =  \frac {k \mathbb J_{d,k}\left(\beta + \frac {1} 2\right)} {\EE \Vol (Z_{\beta,0})},
\qquad
\EE f_{d-k}(Y_{\beta}^\prime) =  k \gamma_{k-1} (\cD_\beta^\prime) =  \frac {k \mathbb J_{d,k}^\prime\left(\beta - \frac {1} 2\right)} {\EE \Vol (Z_{\beta,0}^\prime)}.
$$
\end{theorem}
\begin{proof}
Let us consider the $\beta$-case.
Note that the $\beta$-Voronoi tessellation $\cV_\beta$ is normal by Theorem~10.2.3 of~\cite{SW} (for $\beta=-1$) or by Lemmas~\ref{lem:voronoi_normal} and~\ref{lem:properties_satisfied} (for $\beta>-1$). Hence, Theorem~10.1.2 of~\cite{SW} implies that
$$
\gamma_{d-k} (\cV_{\beta}) = \frac1 {k} \EE f_{d-k}(Y_{\beta}).
$$
By Proposition~\ref{prop:duality_face_intensities}, we also have $\gamma_{k-1}(\cD_{\beta}) = \gamma_{d-k}(\cV_{\beta})$.
Taking these equalities together and recalling Theorem~\ref{theo:cell_intensities} yields the claim in the $\beta$-case. The $\beta^\prime$-case is similar.
\end{proof}

\begin{remark}
For $\beta=-1$, $Y_{-1}$ is the typical cell in the classical Poisson-Voronoi tessellation in $\RR^{d-1}$. The expected $f$-vector of $Y_{-1}$ has been determined in~\cite[Theorem~2.8]{kabluchko_formula}, where $Y_{-1}$ was denoted by $\cV_{d-1}$.  In~\cite{beta_polytopes}, we showed that the expected $f$-vector of $Y_{-1}$ is related to the angle sums of  $\beta^\prime$-simplices, whereas the above theorem expresses it in terms of the values $\mathbb J_{d,k}( - \frac {1} 2)$ originating from $\beta$-simplices.  For the typical Voronoi cell on the sphere, there also exist similar representations in terms of, both, $\beta$- and $\beta^\prime$-simplices~\cite{kabluchko_thaele_voronoi_sphere}.
\end{remark}

\subsection*{Acknowledgement}
{We thank the referees for careful reading and their comments which helped to significantly improve the paper. In particular, we are grateful to a referee who suggested an alternative proof of Lemma \ref{lem:voronoi_normal} which is more elegant than the original one. We would like to thank Claudia Redenbach (Kaiserslautern) for pointing us to the \textit{CGAL Project} in order to create the simulations shown in Figure \ref{fig:beta-tessellations} and Figure \ref{fig:betaprime-tessellations}. AG was partially supported by the the Deutsche Forschungsgemeinschaft (DFG) via RTG 2131 \textit{High-dimensional Phenomena in Probability -- Fluctuations and Discontinuity}.
ZK and CT were supported by the DFG priority program SPP 2265 \textit{Random Geometric Systems}.}

\bibliographystyle{acm}

\end{document}